\documentclass{amsart}

\usepackage{latexsym,amsmath,amsfonts,amscd,amssymb, amsthm,url}

\theoremstyle{plain}
\newtheorem{theorem}{Theorem}[section]
\newtheorem{lemma}{Lemma}[section]
\newtheorem{proposition}{Proposition}[section]
\newtheorem{corollary}{Corollary}[section]

\theoremstyle{definition}
\newtheorem{definition}{Definition}[section]
\newtheorem{example}{Example}[section]

\theoremstyle{remark}
\newtheorem{remark}{Remark}[section]

\title{Symmetries in CR complexity theory}

\author{John P. D'Angelo}

\address{Dept. of Mathematics, Univ. of Illinois, 1409 W. Green St.,
Urbana IL 61801, USA}

\email{jpda@illinois.edu}

\author{Ming Xiao}

\address{Dept. of Mathematics, Univ. of Illinois, 1409 W. Green St.,
Urbana IL 61801, USA}

\email{mingxiao@illinois.edu}

\begin{document}

\maketitle

\begin{abstract} We introduce the Hermitian-invariant group $\Gamma_f$ of a proper rational
map $f$ between the unit ball in complex Euclidean space and a generalized ball in 
a space of typically higher dimension. We use properties of the groups to define
the crucial new concepts of essential map and the source rank of a map. 
We prove that every finite subgroup of the source automorphism group is the Hermitian-invariant group 
of some rational proper map between balls. We prove that $\Gamma_f$ is non-compact if and only if $f$ 
is a totally geodesic embedding. We show that $\Gamma_f$ contains an $n$-torus
if and only if $f$ is equivalent to a monomial map. We show that $\Gamma_f$ contains
a maximal compact subgroup if and only if $f$ is equivalent to the juxtaposition
of tensor powers. We also establish a monotonicity result; the group, after intersecting with the unitary group,
does not decrease when a tensor product operation is applied to a polynomial proper map.
We give a necessary condition for $\Gamma_f$ (when the target
is a generalized ball) to contain automorphisms that move the origin.

\medskip

\noindent
{\bf AMS Classification Numbers}: 32H35, 32H02, 32M99,  30J10, 14P10, 32A50.

\medskip

\end{abstract}

\section{Introduction}

This paper aims to further the development of complexity theory in CR geometry.
Roughly speaking, {\it CR complexity theory}
considers how complicated CR maps between CR manifolds $M$ and $M'$ can be, based on 
geometric information about $M$ and $M'$. A closely related matter considers
the complexity of proper holomorphic mappings between domains with smooth boundaries.

We consider proper rational mappings from the unit ball ${\mathbb B}^n$
in complex Euclidean space ${\mathbb C}^n$ to generalized balls ${\mathbb B}_l^N$
in (typically) higher dimensional spaces. The generalized ball ${\mathbb B}_l^N$ is defined
via a Hermitian form  with $l$ negative eigenvalues. See Definition 2.1.
One measure of complexity of a rational map is its degree. 
We must also consider the source and target dimensions in this discussion.
To do so we introduce and systematically study the notion of Hermitian group invariance.

Given a subgroup $\Gamma$ of ${\rm Aut}({\mathbb B}^n)$, 
we say $f$ is Hermitian $\Gamma$-invariant if, for each $\gamma \in \Gamma$,
there is an automorphism $\psi_\gamma$ of the target (generalized) ball such that 
$  f \circ \gamma =\psi_\gamma \circ f$. 
The {\bf Hermitian invariant group}  $\Gamma_f$ is the maximal such $\Gamma$.
We use this term because we analyze $\Gamma_f$  using Hermitian forms.

We prove in Theorem 3.1 that $\Gamma_f$ is a Lie subgroup of $\mathrm{Aut}(\mathbb{B}^n)$
with finitely many connected components.
In Corollary 3.2 of Theorem 3.2, when the target is a ball,
we show that $\Gamma_f$ is non-compact if and only if
$f$ is a totally geodesic embedding with respect to the 
Poincar\'{e} metric. Otherwise, $\Gamma_f$ is contained in a conjugate of ${\bf U}(n)$.
In Theorem 6.1, when the target is a generalized ball, we give a necessary condition for $\Gamma_f$
to contain an automorphism that moves the origin. 

Theorem 6.2 is perhaps the main result
of this paper. Let $G$ be a finite subgroup of ${\rm Aut}({\mathbb B}^n)$.
Then there is a rational proper map $f$ to some target ball for which $\Gamma_f = G$. Furthermore,
if $G$ is a subgroup of ${\bf U}(n)$, then we may choose $f$ to be a polynomial.
The proof relies on Noether's result that the algebra of polynomials invariant under $G$
is finitely generated. We can prove a related theorem without using finite generation.
If we regard $G$ as a subgroup of the permutation group $S_n$
and represent $S_n$ in ${\bf U}(n)$, then 
we construct in Theorem 6.3 an explicit polynomial proper map $f$ between balls for which $\Gamma_f = G$.

Theorems 5.1 and 5.2 establish the following results: $\Gamma_f$ contains an $n$-torus
if and only if $f$ is equivalent to a monomial map and $\Gamma_f$ contains
a conjugate of the unitary group ${\bf U}(n)$ if and only $f$ is the juxtaposition
of tensor powers. In this paper we assume $f$ is rational.
In a future paper we wll show that Theorems 5.1 and 5.2 remain valid without any regularity assumption.
We also study in detail how the Hermitian group is 
impacted by various constructions of proper maps from
the first author's earlier work. These constructions include the tensor product, a restricted tensor product,
and juxtaposition. See Theorems 4.1, 4.2, Proposition 4.1, and Corollaries 4.1, 4.2.

\begin{example} For $m\ge 2$, the tensor power $z \mapsto z^{\otimes m}$ plays a major role in
this paper. It appropriately generalizes the map $\zeta \mapsto \zeta^m$
in one complex dimension in many important ways. See Definition 1.3 for precise definitions
of the invariant group $G_f$ and the Hermitian invariant group $\Gamma_f$ of a proper map $f$.
For the tensor power, the group $\Gamma_f$  is ${\bf U}(n)$. The group $G_f$ 
is a diagonal unitary representation of a cyclic group of order $m$.
In addition, we use properties of $z^{\otimes m}$
in the proofs of several of our main results. See also Example 7.1.
 \end{example}

The Hermitian invariant group allows us to introduce a new concept, namely the {\it source rank} of a rational map from the unit sphere. We also consider two related notions: {\it image rank} and {\it Hermitian rank}.
These related notions are much easier to understand and have appeared in various guises in many places.

Let $\Gamma \leqslant {\rm Aut}({\mathbb B}^n)$ be a subgroup of the automorphism group of the ball.
We associate with $\Gamma$ a certain integer ${\bf S}(\Gamma)$ as follows.
First we consider subgroups $ G \leqslant \Gamma$ of the form
$$ G = {\bf U}(k_1) \oplus \cdots \oplus {\bf U}(k_J), $$
where $0 \le k_j$ for each $j$ and $\sum k_j \le n$. When $k=0$, we regard ${\bf U}(k)$ as the trivial group.
Define ${\bf S}(G)$ by 
$$ {\bf S}(G)= n - \sum_{k_j \ge 1} (k_j-1). \eqno (1) $$
Note that $1 \le {\bf S}(G) \le n$. 
We define ${\bf S}(\Gamma)$ to be the minimum of ${\bf S}(G)$ for all such $G$. 
For $k_j \ge 1$, the orthogonal summand ${\mathbb C}^{k_j}$ will behave like a one-dimensional space;
we are reducing dimensions from $k_j$ to $1$, hence the definition (1).
The {\bf source rank} of $f$, written ${\bf s}(f)$, will be defined by 
$$ {\bf s}(f) = \min_{\phi} {\bf S}(\phi^{-1} \circ \Gamma_f \circ \phi). \eqno (2) $$
Here the minimum is taken over all $\phi \in {\rm Aut}({\mathbb B}^n)$.

\begin{definition} Let $f:{\mathbb B}^n \to {\mathbb B}^N_l$ be a proper map. Let 
$A_f \leqslant{\rm Aut}({\mathbb B}^n) \times {\rm Aut}({\mathbb B}^N_l)$ be the subgroup consisting
of pairs $(\gamma, \psi)$ for which $f \circ \gamma = \psi \circ f$.  
We let $S$ denote the projection of $A_f$ onto the first factor. 
\end{definition}

In Proposition 3.1, we verify that $A_f$ is a group
and that $S=\Gamma_f$. The next definition introduces the invariant group of $f$
and also provides a direct definition of the Hermitian invariant group.

\begin{definition} Let $f:{\mathbb B}^n \to {\mathbb B}^N_l$ be a proper map.
Let $\Gamma \leqslant {\rm Aut}({\mathbb B}^n)$. 
\begin{itemize}
\item We say $f$ is $\Gamma$-invariant if, for each $\gamma \in \Gamma$,
we have $f\circ \gamma = f$. The maximal such $\Gamma$, written $G_f$, is called
the {\bf invariant group} of $f$.

\item We say $f$ is Hermitian $\Gamma$-invariant if, for each $\gamma \in \Gamma$,
there is an automorphism $\psi_\gamma$ of the target (generalized) ball such that 
$  f \circ \gamma =\psi_\gamma \circ f$. 
The {\bf Hermitian invariant group}  $\Gamma_f$ is the maximal such $\Gamma$.

\item The {\bf source rank} of $f$ is defined by 
$$ {\bf s}(f) = \min_\gamma {\bf S}(\gamma^{-1} \circ \Gamma_f \circ \gamma). \eqno (3) $$
The minimum in (3) is taken over all $\gamma \in {\rm Aut}({\mathbb B}^n)$. The conjugation
by $\gamma$ makes the source rank invariant under composition with automorphisms.

\item The map $f$ is called {\bf source essential} if its source rank equals its source dimension.
It is called {\bf essential} if it is source essential and also its image rank equals its target dimension.
\end{itemize} \end{definition}

\begin{remark} The image rank of $f$ is the smallest $N_0$
for which the image of $f$ lies in an affine space of dimension $N_0$.
One could say (using parallel language) that a map is {\it target essential} if its image rank
equals its target dimension. Since this notion has appeared (with different names) in many papers,
it seems misguided to introduce yet another term. 
We emphasize that the notion of source rank is new. In a future paper,
the authors will study relationships between
these ranks and additional properties of $A_f$.
\end{remark}

A well-known theorem
of Forstneri\v c ([F1]) states, when $n\ge 2$, that a proper holomorphic map $f:{\mathbb B}^n \to {\mathbb B }^N$, assumed (sufficiently) smooth on the sphere,
must be a rational function. In this setting (see [CS]),  the map extends holomorphically past the sphere.

\begin{remark} Definitions 1.1 and 1.2 apply to proper maps without assuming
rationality, but we restrict to the rational case in this paper.
\end{remark}

As noted above, every finite group can be represented as a subgroup of ${\bf U}(n)$ that
is the Hermitian invariant group $\Gamma_f$ of a polynomial proper map $f$ between balls.
By contrast, there are strong restrictions on the invariant group $G_f$ above. For rational proper maps between
unit balls, the group must be cyclic (see [Li]) and the list of possible representations is quite short.  
See [D1] and [D5]. We say a bit more in Section 7.

We now mention the related 
notion of equivariance for holomorphic maps. Assume $f:{\mathbb B}^n \to {\mathbb B}^m$
is holomorphic. Let $\Gamma$ be a subgroup of 
$\mathrm{Aut}(\mathbb{B}^n)$ and let $ \Phi: \Gamma \to \mathrm{Aut}(\mathbb{B}^m)$ be a homomorphism.
Then $f$ is {\it equivariant with respect to}
$\Phi$ if, for each $\gamma \in \Gamma$, we have $f \circ 
\gamma=\Phi(\gamma) \circ f$. (See [KM].) This notion of equivariant map is closely related to a question concerning the rigidity of holomorphic maps between compact 
hyperbolic spaces raised by Siu [S].  For more details and
results, see [CM]. See also [Hu] for a CR geometric version of the question. 

In our setting, $f \circ \gamma = \psi_\gamma \circ f$ for some $\psi_\gamma$, but we do not assume any properties of the relationship between $\gamma$ and $\psi_\gamma$. When the target is also a ball,
formula (12) from Proposition 3.3 provides a computational method 
for deciding whether such a $\psi_\gamma$ exists. This method arises throughout the paper,
as it makes a precise connection to the use of Hermitian forms in CR geometry.
Furthermore, in many situations arising in this paper,
given $\gamma$, there are many possible $\psi_\gamma$ with $f \circ \gamma = \psi_\gamma \circ f$.

This paper includes many examples. In Section 7 we compute the Hermitian invariant group for
all the proper rational maps from ${\mathbb B}^2$ to ${\mathbb B}^3$. In Section 4 we compute the Hermitian
norm ${\mathcal H}(f)$ when $f$ is a tensor product of automorphisms. 
In Section 8 we relate our results to CR complexity and the degree estimate conjecture.

The first author acknowledges support from NSF Grant DMS-1361001. 
The authors also acknowledge AIM. Workshops there in 2006, 2010, and 2014 
all helped in the development of the subject of CR complexity. The
authors thank Peter Ebenfelt, Xiaojun Huang, and Ji\v{r}\'\i\ Lebl
for useful conversations about related matters. The second author also thanks
Sui-Chung Ng and Xin Zhang for helpful discussions. Both authors thank the referee
for suggesting a shorter introduction.

\section{preliminaries}

In this paper, $||z||^2 = \sum_{j=1}^n |z_j|^2$ denotes the squared norm 
in complex Euclidean space of some often unspecified dimension. We write
$$ ||z||_l^2 = \sum_{j=1}^m |z_j|^2 - \sum_{j=m+1}^{m+l} |z_j|^2 \eqno (4) $$ 
for the Hermitian form on ${\mathbb C}^{m+l}$ with $l$ negative eigenvalues.
We write $\langle z,w\rangle$ for the usual Hermitian inner product in ${\mathbb C}^n$ 
and $\langle z,w \rangle_l$ for the sesquilinear form corresponding to formula (4).
Let ${\mathbb P}^N$ denote complex projective space. 

\begin{definition} The generalized unit ball ${\mathbb B}_l^N$ 
is the subset of ${\mathbb P}^N$ consisting of those $\{(z,\zeta)\}$ for which $||z||_l^2 < |\zeta|^2$.
\end{definition}

We make crucial use of the automorphism groups
of these generalized balls. Let ${\bf{U}}(n)$ denote the unitary group on ${\mathbb C}^n$ and let
${\bf U}(m,l)$ denote the linear maps on ${\mathbb C}^{m+l}$ that preserve the form $||z||_l^2$.
As usual, ${\bf SU}(m,l)$ denotes the maps in ${\bf U}(m,l)$ of determinant $1$.
 
The holomorphic automorphism group ${\rm Aut}({\mathbb B}^n)$ of the unit ball 
is the quotient of the real Lie group ${\bf {SU}}(n,1)$ by its center. The center consists of multiples
of the identity operator by $e^{i\theta}$, where $e^{i(n+1)\theta}=1$ to guarantee
that the determinant equals $1$. Each element of ${\rm Aut}({\mathbb B}^n)$ can be 
written $U\phi_a$, where $U \in {\bf U}(n)$ 
and $\phi_a$ is the linear fractional transformation defined by
$$ \phi_a(z) = {a - L_a z \over 1- \langle z,a \rangle}. \eqno (5) $$
Here $a$ is in the unit ball, $s^2 = 1 - ||a||^2$, and 
$$ L_a(z) = {\langle z,a\rangle a \over s+1} + s z. $$ 
The group ${\rm Aut}({\mathbb B}^n)$ is transitive.
Proper maps $f,g$ are {\it spherically equivalent} if
there are automorphisms $\phi$ in the domain and $\chi$ in
the target such that $f = \chi \circ g \circ \phi$.

The story is similar for the generalized balls. 
The holomorphic automorphism group ${\rm Aut}({\mathbb B}_l^ n)$
is the quotient of the real Lie group ${\bf {SU}}(n,l+1)$ by its center.
The automorphism group again consists of linear fractional transformations; it includes
linear maps and maps that move the origin.

Given a rational mapping ${p \over q}$, 
we assume without loss of generality
that the fraction is reduced to lowest terms and $q(0)=1$. We make this normalizing convention
throughout, often without comment.

\begin{definition} Let $f={p \over q}$ be a proper rational map from ${\mathbb B}^n$ to ${\mathbb B}_l^N$ that is normalized as above.
Put $N=m+l$. We define ${\mathcal H}_l(f)$ by:
$$ {\mathcal H}_l(f) = \sum_{j=1}^{m} |p_j|^2 - \sum_{j=m+1}^{m+l} |p_j|^2 - |q|^2. $$
The {\bf Hermitian rank} of $f$ is the rank of the Hermitian form ${\mathcal H}_l(f)$. 
\end{definition}

We can regard ${\mathcal H}_l(f)$ both as a Hermitian form on a vector space of polynomials and
as a real-valued polynomial. Note that ${\mathcal H}_l(f)$ vanishes on the unit sphere; since the ideal
of real polynomials vanishing on the unit sphere is principal, we see that
there is a real polynomial $u(z,{\overline z})$, called the {\it quotient form}, such that
$${\mathcal H}_l(f) = u(z,{\overline z}) \ (||z||^2 - 1). $$  
To be consistent with our other notation, we write ${\mathcal H}(f)$ for ${\mathcal H}_0(f)$.

Assume $f = {p \over q}$ is a rational function 
sending the unit sphere $\mathbb{S}^{2n-1}$ into some $\mathbb{S}^{2N-1}$.
For $k<N$, the unit sphere in ${\mathbb C}^N$ contains many isometric images of 
the unit sphere in ${\mathbb C}^k$. If the image of $f$ happens to lie in one of these spheres,
then we can think of $f$ as a map to $\mathbb{S}^{2k-1}$. To do so, 
consider the Hermitian form given by  
$$ {\mathcal H}(f) = ||p||^2 -|q|^2. \eqno (6) $$
We will use many times the formula
$$ {\mathcal H}(\psi \circ f) = c_\psi {\mathcal H}(f). \eqno (7) $$
In other words, if we compose $f$ with a target automorphism, then the form
${\mathcal H}(f)$ gets multiplied by a positive constant. Hence its signature
is invariant.  This form will have $k$ positive eigenvalues and $1$ negative eigenvalue.
We call $k+1$ the {\bf Hermitian rank} of $f$.  Consider also the 
smallest integer for which the image of the unit ball
in the domain is contained in a $k$-dimensional affine subspace of the target space. 
We call this number the {\bf image rank} or the {\bf embedding dimension} of $f$. For maps to spheres
the Hermitian rank always equals $1$ plus the image rank.

When the image rank of $f$ is $k$, we can regard 
$f$ as mapping $\mathbb{S}^{2n-1}$ to some $\mathbb{S}^{2k-1}$.
Thus the Hermitian form governs {\it target complexity}.
Our discussion of complexity in {\bf both} the source and target considers 
Hermitian forms and corresponding subgroups of the automorphism groups.

Next we motivate the more elusive notion of source rank.
Let $p$ be a polynomial mapping the unit sphere in ${\mathbb C}^n$ to the unit sphere in ${\mathbb C}^N$.
Then $||p(z)||^2 - 1$ is divisible by $||z||^2-1$. To give a simple example, suppose 
$$ ||p(z)||^2 -1 = g(||z||^2) \ (||z||^2 -1) \eqno (8.1) $$
for a polynomial $g$ in one real variable. Then there is a polynomial map
$p^*$ of {\bf one variable} that maps the circle to a sphere; for $\zeta \in {\mathbb C}$, this polynomial satisfies
$$ ||p^*(\zeta)||^2 -1 = g(|\zeta|^2) \ (|\zeta|^2 -1). \eqno (8.2) $$
In this case the {\bf source rank} of $p$ is $1$.
Many higher dimensional maps in fact have source rank $1$; examples include all unitary maps
and tensor products of unitary maps. The key point is that the Hermitian form ${\mathcal H}(p)$
associated with $p$ is invariant under a large group. In (8.1), the group is ${\bf U}(n)$. 

We write $z \oplus w$ to denote 
the orthogonal sum of $z \in {\mathbb C}^k$ and $w \in {\mathbb C}^l$.
Let $G$ and $H$ be groups of transformations on ${\mathbb C}^k$ and ${\mathbb C}^l$.
We write $G \oplus H$ for the group of transformations $g \oplus h$, where
$$ (g \oplus h) (z \oplus w) = g(z) \oplus h(w). $$

Given a subgroup $\Gamma \leqslant {\rm Aut}({\mathbb B}^n)$ 
and an element $\gamma \in {\rm Aut}({\mathbb B}^n)$, as usual 
we call $\gamma^{-1} \circ \Gamma  \circ \gamma$ a {\bf conjugate}
of $\Gamma$.

Given a rational mapping ${p \over q}$, also with target space ${\mathbb C}^{m+l}$,
its associated Hermitian form ${\mathcal H}(f)$ is defined by
$$ ||p||_l^2 - |q|^2 = \sum_{j=1}^m |p_j|^2 - \sum_{j=m+1}^{m+l} |p_j|^2 -|q|^2. $$
We refer to [BEH], [BH], and their references for rigidity results about mappings to generalized balls.

We consider some examples when $l=0$.
Let $f = {p \over q}$ be a rational proper map between balls with $f(0)=0$.
Let $\psi_a$ be an automorphism of the source ball 
with $\psi_a(0)=a$. 
Write $F={P \over Q} = \psi_a \circ f$. Then we have
$$ {\mathcal H}(F) = ||P||^2 - |Q|^2 = (1- ||a||^2) (||p||^2 - |q|^2) = 
 (1- ||a||^2) {\mathcal H}(f). \eqno (9) 
$$
Thus ${\mathcal H}(F)$ is a constant times ${\mathcal H}(f)$. 
See [L] for the following application: in source dimension at least two,
a quadratic rational proper map between balls is spherically 
equivalent to a quadratic monomial map. See also [JZ].

The notion of source rank is elaborated in the next two examples.

\begin{example} Consider any monomial proper map with source dimension $2$. First
write the variables as $(z,w)$. Thus
$$ p(z,w) = (..., c_{ab}z^a w^b, ...). $$
The condition for sending the sphere to a sphere is that $|z|^2 + |w|^2 =1 $ implies
$$ \sum |c_{ab}|^2 |z|^{2a} |w|^{2b} = 1. $$ 
This condition is linear in the positive numbers $|c_{ab}|^2$.
We can create a related map with larger source and target dimensions as follows.
Formally replace $z$ by ${\bf z} = (z_1,..., z_k)$ and $w$ by ${\bf w} = (w_1,...,w_l)$; replace
$|z|^2$ by $||{\bf z}||^2$ and $|w|^2$ by $||{\bf w}||^2$. 
Replace $z^a w^b$ by ${\bf z}^{\otimes a} \otimes {\bf w}^{\otimes b}$.
We obtain a polynomial map $P({\bf z}, {\bf w})$ with source dimension $k+l$. We regard
such maps as {\it inessential}; although $P$ has source dimension $k+l$, its source rank is at most $2$. 
The Hermitian-invariant group $\Gamma_P$ contains ${\bf U}(k) \times {\bf U}(l)$. By
(1), we obtain ${\bf s}(P) \le (k+l)- (k-1) - (l-1) = 2$. 
\end{example}

\begin{example} Suppose that the quotient form $u(z,{\overline z})$ of $f$, defined just after
Definition 2.2, can be written
$$ u(z,{\overline z}) = u(||z'||^2, ||z''||^2, z''', {\overline z'''}), $$
where $z= (z',z'', z''')$ and $z'$ has $n_1$ components, $z''$ has $n_2$ components, and $z'''$
has $n_3$ components. Then ${\bf s}(f) \le 2 + n_3$. Informally we regard 
the $z'$ and $z''$ variables as one-dimensional. 
\end{example}

\section{Proper maps and invariant groups}

Let $f: {\mathbb B}^n \to {\mathbb B}^N_l$ be a proper rational map.
Recall that $A_f \subseteq {\rm Aut}({\mathbb B}^n) \times {\rm Aut}({\mathbb B}^N_l)$ consists
of the pairs $(\gamma, \psi)$ for which 
$f \circ \gamma = \psi \circ f$. Let $S$ be the projection of $A_f$ onto the first factor and $T$
the projection of $A_f$ onto the second factor.

Thus $S$ is the subset consisting
of those source automorphisms $\gamma$ for which there is a target automorphism $\psi_\gamma$ with
$$  f \circ \gamma = \psi_\gamma \circ f. \eqno (10) $$

\begin{proposition} Let $f: {\mathbb B}^n \to {\mathbb B}^N_l$ be a proper rational map.
Let $A_f, S, T$ be as above. Then $A_f$ is a subgroup 
of ${\rm Aut}({\mathbb B}^n) \times {\rm Aut}({\mathbb B}^N_l)$, 
and both $S$ and $T$ are subgroups of the corresponding automorphism groups.
\end{proposition}

\begin{proof} First we check that $A_f$ is a subgroup. If
$(\gamma, \psi) \in A_f$, then $(\gamma^{-1}, \psi^{-1}) \in A_f$, because
$f \circ \gamma = \psi \circ f$ implies
$$ \psi^{-1} \circ f = f \circ \gamma^{-1}. \eqno (11.1) $$
Furthermore, if $(\gamma_1, \psi_1)$ and $(\gamma_2, \psi_2)$ are in $A_f$, then
$(\gamma_1 \circ \gamma_2, \psi_1 \circ \psi_2) \in A_f$ because 
$$ \psi_1 \circ \psi _2 \circ f =
\psi_1 \circ f \circ \gamma_2 = f \circ \gamma_1 \circ \gamma_2. \eqno (11.2) $$
Hence $A_f \leqslant {\rm Aut}({\mathbb B}^n) \times {\rm Aut}({\mathbb B}^N_l)$.
Both projections are homomorphisms; thus each of $S$ and $T$ is a subgroup of its corresponding automorphism group.
 \end{proof}

Note that $S=\Gamma_f$, because $\Gamma_f$ is the maximal Hermitian invariant subgroup.
We do not make any particular use of the target group $T$ 
from this proposition, but we note in passing that the image
of the ball is preserved under maps in $T$. In the remainder
of the paper we will regard $\Gamma_f$ as defined by (10).

The next result follows immediately from the definitions. 
The subsequent result, Proposition 3.3, requires that the target be a 
ball.

\begin{proposition} 
Let $f,g$ be rational proper maps from ${\mathbb B}^n$  to ${\mathbb B}^N_l$. Assume that
there are automorphisms for which
$$\psi \circ f=g\circ \varphi.$$
Let $\Gamma_f, \Gamma_g$ be the Hermitian invariant groups of $f$ and $g$.
Then $\Gamma_f$ and $\Gamma_g$ are conjugate by $\varphi$:
$$\Gamma_f =\varphi^{-1} \circ \Gamma_g \circ \varphi.$$
\end{proposition}

\begin{proposition} Let $f={p \over q}$ be a proper rational map from ${\mathbb B}^n$ to some
${\mathbb B}^N$. Let $\Gamma_f \leqslant {\rm Aut}({\mathbb B}^n)$ be the Hermitian invariant group of $f$.  
Then $\gamma \in \Gamma_f$ if and only if 
there is a constant $c_\gamma$ such that 
$$ {\mathcal H} (f \circ \gamma) = c_\gamma {\mathcal H}(f). \eqno (12) $$
\end{proposition}

\begin{proof} First assume that (12) holds. Write $f = {p \over q}$ and $f \circ \gamma = {P \over Q}$.
By convention we assume $q(0)=Q(0)=1$ and that the fractions are in lowest terms.
After composing with automorphisms of the target we may also assume $p(0)=0$ and $P(0)=0$.
Our assumption (12) yields
$$ ||P||^2 - |Q|^2 = C_\gamma ( ||p||^2 - |q|^2), \eqno (13) $$
and thus the constant $C_\gamma$ must equal $1$. Write $Q=1+A$ and $q=1+a$ and plug in (13).
Equating pure terms yields
$$ 2 {\rm Re} (A) = 2 {\rm Re}(a).$$
Since $A,a$ are polynomials vanishing at $0$ we obtain $A=a$. Equating mixed terms then gives
$$ ||P||^2 - ||A||^2 = ||p||^2 - |a|^2, $$
and hence $||P||^2 = ||p||^2$. Therefore $P=Up$ for some unitary $U \in {\bf U}(N)$.
It follows that $f \circ \gamma = g_\gamma \circ f$ for some automorphism $g_\gamma$, and Definition 1.3 holds.

The converse is easy: we are given Hermitian invariance and we need to prove (12). 
If $\varphi$ is an automorphism of the target ball,
then $\varphi = U \circ \phi_a$ where $\phi_a$ satisfies (5). We may assume $f(0)=0$. Hence, by (9),
$$ {\mathcal H}(\varphi \circ f) = (1- ||a||^2) {\mathcal H}(f). \eqno (14) $$
The equality $ f \circ \gamma =\psi_\gamma \circ f$ and 
(14) guarantee that 
$$ {\mathcal H}(f \circ \gamma ) = {\mathcal H}( \psi_\gamma \circ f) = 
c_\gamma {\mathcal H}(f) $$
for a non-zero constant $c_\gamma$. Hence (12) holds. 
\end{proof}

\begin{remark} When $\gamma \in {\bf U}(n)$, the 
constant $c_\gamma$
from (12) necessarily equals $1$.
\end{remark}

\begin{corollary} For any proper rational map $f$ between balls, $\Gamma_f = \Gamma_{f\oplus 0}$. \end{corollary}

We give a simple example computing ${\mathcal H}(f)$ and illustrating source rank.

\begin{example} Put $f(z) = ( z_1,z_2, z_1z_3, z_2z_3, z_3^2)$.
Then $f:{\mathbb B}^3 \to {\mathbb B}^5$ is proper and
$$ {\mathcal H}(f) = |z_1|^2 + |z_2|^2 + (|z_1|^2 + |z_2|^2 + |z_3|^2) |z_3|^2 - 1. $$
This Hermitian form is invariant under the subgroup ${\bf U}(2) \oplus {\bf U}(1)$ of ${\bf U}(3)$
and ${\bf s}(f) =2$. In this case $f$ itself 
is invariant only under the trivial group. \end{example}

\begin{theorem}
Let $f$ be a rational proper map from $\mathbb{B}^n$ to $\mathbb{B}^N.$ 
Then $\Gamma_f$ is a Lie subgroup of $\mathrm{Aut}(\mathbb{B}^n)$
with finitely many connected components.
\end{theorem}

\begin{proof}
By composing $f$ with an automorphism of $\mathbb{B}^N$ if necessary, we 
may assume $f(0)=0$.
Regard $\mathbb{B}^n$ and $\mathbb{B}^N$ as open subsets of $\mathbb{P}^n$ and $\mathbb{P}^N.$ Then $\Gamma_f$ can be regarded as a subgroup of ${\bf SU}(n, 1)/Z$; here $Z$ is the center of ${\bf SU}(n, 1)$.
Write $[z, s]$ for homogeneous coordinates of $\mathbb{P}^n$. In homogeneous coordinates, write $f$ as
$$\hat{f}=[p_1(z, s), \cdots, p_N(z, s), q(z, s)].$$
Here $p_1, \cdots, p_N, q$ are homogeneous polynomials in $z, s$ with no 
nonconstant common factor. 
We can assume $p_i(0, 1)=0$ for each $1 \leq i \leq N$ and $q(0,1)=1.$ 

Let $\gamma \in \mathrm{Aut}(\mathbb{B}^n)$. Write
$\gamma=U=(u_{ij})_{1 \leq i, j \leq n+1} \in {\bf SU}(n, 1)$.
Note  that
$$f \circ \gamma= [p_1((z, s)U), \cdots, p_N((z, s)U), q((z, s)U)].$$
For $1 \leq i \leq N$ write $p_i((0,1)U)=\lambda_i(U)$ and $q((0,1)U) =\mu(U)$.
Define $\lambda(U)$ by
$$\lambda(U)=\sum_{i=1}^N |\lambda_i(U)|^2- |\mu(U)|^2 <0. $$ 
Then $\lambda(U)$  depends polynomially on the entries $u_{ij}$.

Note that $\gamma \in \Gamma_f$ if and only if 
$$||p((z,s)U)||^2-|q((z, s)U)|^2=- \lambda(U) \left( ||p(z, s)||^2- |q(z, s)|^2\right). $$
For every choice of multi-indices $\alpha, \beta,$ and nonnegative integers $\mu, \nu$, we equate terms of the form $z^{\alpha}s^{\mu}\overline{z}^{\beta}\overline{s}^{\nu}$ in this equation. We obtain a system of real-valued polynomial
equations in the $u_{ij}$ whose
solution set is precisely $\Gamma_f$. 
Thus $\Gamma_f$ is a real-algebraic subvariety of ${\bf SU}(n, 1)/Z$.
Hence it is a closed subset of ${\rm Aut}({\mathbb B}^n)$ 
and has finitely many connected components. Since $\Gamma_f$ is closed and is a subgroup of
${\rm Aut}({\mathbb B}^n)$ by Proposition 3.1, it is a Lie subgroup.
\end{proof}

The following result often enables us to work only with unitary automorphisms
when studying Hermitian invariance for polynomial proper maps. In addition, as a corollary,
the Hermitian invariant group of a rational proper map between balls 
can be non-compact only when the map is totally geodesic. 

\begin{theorem} Let $f={p \over q}: {\mathbb B}^n \to {\mathbb B}^N$ be a rational proper map of degree $d$
with $p(0)=0$ and $q(0)=1$.
Suppose $\Gamma_f$ contains an automorphism $\gamma= U\phi_a$ that moves the origin. Then
$$ (||p(a)||^2 - |q(a)|^2 ) \ (||p(Ua)||^2 - |q(Ua)|^2) = (1 - ||a||^2)^{2d}.  \eqno (15) $$
Furthermore, if $f=p$ is a polynomial, then 
$p$ is unitarily equivalent to $z \oplus 0$. 
\end{theorem}

\begin{proof} First (see [D1]) the degree of $q$ is at most $d-1$. 
Write 
$$ p(z) = \sum_{|\alpha|=1}^d A_\alpha z^\alpha $$ 
$$ q(z) = \sum_{|\beta|=0}^{d-1} b_\beta z^\beta. $$
Here $b_{\bf 0}=1$. Assume there is an automorphism $\gamma \in \Gamma_f$ which moves the origin. By (5)
we have $\gamma= U \phi_a$. Note that $L_a(a) = a$. Proposition 3.3 guarantees
the existence of a constant $c_\gamma$, which we compute in two ways.
This constant satisfies
 $$ c_\gamma (||p||^2 - |q|^2) = c_\gamma {\mathcal H}(f) = {\mathcal H} (f \circ (U \phi_a)). $$ 
Put $f \circ (U \phi_a)= {P \over Q}$ where,
by our convention, $Q(0)=1$. By definition, 
$$ {\mathcal H} (f \circ (U \phi_a)) = ||P||^2 - |Q|^2. $$ 
Using formula (5) for $\phi_a$ gives a formula for $||P||^2 - |Q|^2$.
Since it is long we write 
$$ ||P(z)||^2 = {1 \over |q(a)|^2}  \left(|| \sum _{|\alpha|=1}^d A_{\alpha} (U(a-L_a(z))^\alpha (1- \langle z, a \rangle)^{d-|\alpha|}||^2 \right) \eqno (16.1) $$
$$ |Q(z)|^2 = {1 \over |q(a)|^2}  \left( \big| \sum _{|\beta| =0}^{d-1} b_{\beta} (U(a-L_a(z))^\beta (1- \langle z, a \rangle)^{d-|\beta|}\big|^2 \right). \eqno (16.2) $$

The factor of ${1 \over |q(a)|^2}$ arises in order to make $Q(0)=1$.
We evaluate (16.1) and (16.2) at $0$ to get
$$ ||P(0)||^2 - |Q(0)|^2 =  {1 \over |q(a)|^2} \ (||p(Ua)||^2 - |q(Ua)|^2). \eqno (17.1)  $$
Then we evaluate them at $a$, using $L_a(a) = a$, to get
$$ ||P(a)||^2 - |Q(a)|^2 = {1 \over |q(a)|^2} (- (1 - ||a||^2)^{2d}). \eqno (17.2) $$
Evaluating ${\mathcal H}(f)$ at $0$ and $a$, and using
$c_{\gamma} {\mathcal H}(f) = {\mathcal H}(f \circ \gamma)$,  yields both formulas
$$ c_\gamma(||p(0)||^2 -1) = - c_\gamma = {1 \over |q(a)|^2} \left(||p(Ua)||^2 - |q(Ua)|^2 \right) $$
$$ c_\gamma(||p(a)||^2 -|q(a)|^2) = - {1 \over |q(a)|^2} (1-||a||^2)^{2d}. $$
Formula (15) follows.

Let $f=p$ be a polynomial proper map with $p(0)=0$.
By Schwarz's lemma, $||p(z)||^2 \le ||z||^2$ on the ball. Since $||Ua||^2 = ||a||^2$, we get
$$ (1- ||p(a)||^2) \ (1- ||p(Ua)||^2) \ge (1 - ||a||^2)^2. $$
Plugging this inequality in (15) yields
$$ (1 - ||a||^2)^{2d} \ge (1- ||a||^2)^2, $$
which, since $a \ne 0$, can hold only if $d=1$. Hence $p$ is linear and the conclusion follows.
\end{proof}

\begin{corollary} Let $f:{\mathbb B}^n \to {\mathbb B}^N $ be a rational proper map.
Then $\Gamma_f$ is 
non-compact if and only if
$f$ is totally geodesic with respect to the Poincar\'{e} metric.  Otherwise,
$\Gamma_f$ lies in a maximal compact subgroup, i. e., a conjugate of ${\bf U}(n)$. \end{corollary}
\begin{proof} Let $f= {p \over q}$ be of degree $d$. After composition with an automorphism
of the target, we may assume $f(0)=0$.
Assume $U\phi_a \in \Gamma_f$. In this case, Schwarz's lemma yields $||p(z)||^2 \le |q(z)|^2 \ ||z||^2$
for $z$ in the ball. Therefore 
$$ |q(z)|^2 - ||p(z)||^2 \ge |q(z)|^2 ( 1- ||z||^2). $$
 Since $||Ua||^2 = ||a||^2$, we plug this inequality into (15) to get
$$ (1- ||a||^2)^{2d} \ge (1- ||a||^2)^2 |q(a)|^2 |q(Ua)|^2 $$
and therefore
$$ (1- ||a||^2)^{2d-2} \ge |q(a)|^2 |q(Ua)|^2. \eqno (18) $$

If $\Gamma_f$ is not compact, then we can find a sequence of 
automorphisms $U_k \phi_{a_k} \in \Gamma_f$
where $||a_k||$ tends to $1$. Assume $d\ge 2$.
Then (18) implies that a subsequence of $q(a_k)$ or of $q(U_ka_k)$
tends to $0$. But, by [CS], the denominator $q$ 
cannot vanish on the closed ball. This contradiction therefore implies $d=1$. The degree of $q$
is smaller than the degree of $p$. Thus, if the degree of $p$ is $1$,
then $q$ is constant. Hence $f$ is linear and therefore totally geodesic.

Assume $f$ is not totally geodesic.
By Theorem 3.1, $\Gamma_f$ is closed in ${\rm Aut}({\mathbb B}^n)$. Thus $\Gamma_f$ is compact.
By standard Lie group theory (see [HT]), $\Gamma_f$
is contained in a maximal compact subgroup which must be a conjugate of ${\bf U}(n)$. 
\end{proof}

\begin{corollary} Let $p$ be a proper polynomial map between balls with
$p(0)=0$. Unless $p$ is of degree $1$, we have $\Gamma_p \subseteq {\bf U}(n)$. \end{corollary}

\section{operations on rational proper maps between balls}

Given one or more rational proper maps, one can create additional proper maps
via certain constructions. See [D1], [D3], [D5].
We recall these constructions in this section and determine 
how they impact Hermitian invariance.

\begin{example} Assume $f$ and $g$ are rational proper maps with the same source.
Their {\bf tensor product} $f \otimes g$ is then a rational proper map. \end{example}

We also have a more subtle restricted tensor product operation.

\begin{example} Assume $f$ and $g$ are rational proper maps with the same source.
Let $A$ be a subspace of the target of $f$. 
Let $\pi_A$ denote orthogonal projection onto $A$. We define $E_{A,g} f$ by 
$$  E_{A,g}f = ((\pi_A f) \otimes g) \oplus (1- \pi_A)f. \eqno (19) $$
Then $E_{A,g}$ is a rational proper map.
When $g$ is the identity map, we simply write $E_Af$ or even $Ef$ when $A$ is fixed. 
In this case we call $Ef$ a {\bf first descendant} of $f$. \end{example}

We can also combine two proper maps via a process of {\bf juxtaposition}.

\begin{example} Assume $f$ and $g$ are rational proper maps with the same source, 
and $0 < \theta < {\pi \over 2}$.
The $\theta$-juxtaposition of $f$ and $g$, written $J_\theta(f,g)$, is defined by
$$ J_\theta(f,g) = \cos(\theta)f \oplus \sin(\theta)g. \eqno (20) $$ 
For each $\theta$, the map $J_\theta(f,g)$ is also a rational proper map. \end{example} 

The family of maps in Example 4.3 defines a homotopy from $f \oplus 0$ to $0 \oplus g$. 

\begin{example} We can juxtapose any finite
collection of proper maps (with the same source) as follows. Suppose $\lambda \in {\mathbb C}^k$ and
$||\lambda||^2 = 1$. We define

$$ J_\lambda(g_1,...,g_k) = \lambda_1 g_1 \oplus \lambda_2 g_2 \oplus ... \oplus \lambda_k g_k. \eqno (21) $$
Again, $J(g_1,...,g_k)$ is a rational proper map if each $g_j$ is. \end{example}

\begin{remark} Another construction is illustrated in Example 2.2, but this
process creates a larger source dimension as well as a larger Hermitian-invariant group. \end{remark}

We naturally wish to study the impact of the tensor operation from Example 
4.2 on Hermitian invariance. An elegant result holds  when the subspace 
$A$ is chosen as described below. See Example 4.5 for some possibilities
when $A$ is chosen otherwise.  

Let $f$ be a polynomial proper map of degree at least $2$.
We study the relationship between $\Gamma_f$ and $\Gamma_{Ef}$.
Write $f = \sum_{j=\nu}^d f_j$, where $f_j$ is homogeneous of degree $j$ and vector-valued.
Assume that $f$ itself is not homogeneous, that is, $\nu < d$. 
The condition that $f$ is a proper map forces various identities, one of which we need here.
Let $A$ denote the subspace of ${\mathbb C}^N$ 
spanned by the coefficient vectors of $f_\nu$.
Then the coefficients of $f_d$ are all orthogonal to $A$. As a result, we 
can define 
the descendant $Ef$ as in (19). It remains of degree $d$, but its order of vanishing $\nu$ has increased.
After finitely many steps, we obtain a descendant that is homogeneous, and 
hence linearly equivalent to $z^{\otimes d} \oplus 0$. (See [D1].)

Theorem 4.1 shows what happens to $\Gamma_f$ when we replace $f$ by $Ef$.
The intersection of the Hermitian invariant group with the unitary group does not decrease
under this tensor product operation. 
The intersection with ${\bf U}(n)$ in the 
statement is required for the following reason.
If $f$ is unitary, then $\Gamma_f$ is the full automorphism group. 
If we form the descendant $Ef$, where
$A$ equals the full target space,  then $\Gamma_{Ef}={\bf U}(n)$, 
and the group is smaller.  By intersecting with ${\bf U}(n)$, 
we avoid this situation. 

\begin{theorem} Let $f:{\mathbb C}^n \to {\mathbb C}^N$ be a 
polynomial proper map. Let $A$ be the subspace spanned by the 
coefficient vectors of the lowest order part of $f$.
Write  $Ef$ for $E_{A,z}f$, as defined in (19).
Then
$$ \Gamma_f \cap {\bf U}(n) \subseteq \Gamma_{Ef}.  $$\end{theorem}

\begin{proof} First note that $Ef(0)=0$; by Corollary 3.3, 
$\Gamma_{Ef} \subseteq {\bf U}(n)$. Next, if $f$ is homogeneous, then it is unitarily equivalent to
$z^{\otimes d}$. (See [D1].) In this case $A$ must be the full space, and $Ef= z^{\otimes (d+1)}$.
The conclusion holds because 
$$ \Gamma_f \cap {\bf U}(n) = {\bf U}(n) = \Gamma_{Ef}.   $$ 
Next suppose $f$ is not homogeneous. Put
$f = f_\nu + ... + f_d$, where $\nu < d$. Let $A$ be the span
of the coefficients of $f_\nu$. Formula (19) holds. Let $\gamma \in \Gamma_f \cap {\bf U}(n)$. Then
${\mathcal H}(f \circ \gamma) = {\mathcal H}(f)$. 
Put $\rho = ||z||^2 -1$. Then
$$ {\mathcal H}(Ef) = ||Ef||^2 -1 = ||f||^2 -1  + ||\pi_Af||^2 \ \rho = {\mathcal H}(f) + ||\pi_Af||^2 \ \rho, \eqno (22) $$
and for each unitary $\gamma$ we have
$$ {\mathcal H}(Ef \circ \gamma) = ||f \circ \gamma||^2 -1 + ||(\pi_Af) \circ \gamma||^2 \ \rho $$
$$ = 
{\mathcal H}(f\circ \gamma) + ||(\pi_Af) \circ \gamma||^2 \rho = {\mathcal H}(f) + ||(\pi_Af) \circ \gamma||^2 \rho. \eqno (23) $$
To finish the proof we need to show that the term $||(\pi_Af) \circ \gamma||^2$ in (23)
can be replaced with $||\pi_Af||^2$. 
We establish this equality as follows.

First we claim that $(\pi_Af) \circ \gamma = \pi_A(f \circ \gamma)$. 
Fix an orthonormal basis $\{e_k\}$ of $A$ and write 
$$ \pi_A f = \sum \langle f, e_j \rangle e_j. \eqno (24) $$
Replacing $f$ by $f \circ \gamma$ gives 
$$ \pi_A (f \circ \gamma) = \sum \langle f\circ \gamma, e_j \rangle e_j = (\pi_A f) \circ \gamma. \eqno (25) $$
Since ${\mathcal H}(f \circ \gamma) = {\mathcal H}(f)$, 
there is a unitary $U$ with $Uf = U \circ f =  f \circ  \gamma$.
Since $U$ is unitary, it preserves both orthonormal bases and inner products. Thus
$$ \pi_{UA} (Uf) = \sum \langle Uf, U(e_j)\rangle (Ue_j) = 
\sum \langle f, e_j \rangle (Ue_j) =  U \left(\sum \langle f, e_j \rangle e_j \right). $$
Hence $\pi_{UA}(Uf) = U (\pi_{A} f)$.
To finish, we need to note that $UA=A$. But 
$$A= \mathrm{Span}_{\mathbb{C}}\{f_{\nu}(z): z \in 
\mathbb{C}^n\}=\mathrm{Span}_{\mathbb{C}}\{f_{\nu}(\gamma(z)): z \in 
\mathbb{C}^n\}=$$
$$ \mathrm{Span}_{\mathbb{C}}\{Uf_{\nu}(z): z \in 
\mathbb{C}^n\}=UA.$$
We used of course that $f_\nu \circ \gamma = U f_\nu$. 
Again using  $Uf = f \circ  \gamma$ we obtain
$$ (\pi_A f) \circ \gamma = \pi_A(f \circ \gamma) = \pi_{A}(Uf) = \pi_{UA}(Uf) = U \pi_A f. $$
Therefore $ ||(\pi_A f) \circ \gamma||^2 = ||\pi_A f||^2 $
as required.
\end{proof}

\begin{corollary} Suppose $f$ is as in Theorem 4.1 and $f(0)=0$.
If, in addition, $f$ is also of degree at least $2$, then
$$ \Gamma_f \subseteq \Gamma_{Ef}. $$\end{corollary}
\begin{proof} By Theorem 3.2 or Corollary 3.3, both groups lie in ${\bf U}(n)$. \end{proof}

\begin{example} First consider the identity map $p$
given by $p(z_1,z_2)= (z_1, z_2)$ in source dimension 
$2$. Its invariant group is the full automorphism group.
Let $A$ be the subspace spanned by $(0,1)$. Then $E_Ap$
is the Whitney map given by 
$$ E_Ap(z_1,z_2) = f(z_1,z_2) = (z_1, z_1 z_2, z_2^2) $$ and
$\Gamma_f = {\bf U}(1) \oplus {\bf U}(1)$. See Proposition 7.1.
Next let $A$ be the span 
of $(1,0,0)$, arising from the lowest order part of $f$ as in Theorem 4.1. 
Then $Ef$ is equivalent to the map $(z_1^2, \sqrt{2}z_1z_2, z_2^2)$, and hence
 $\Gamma_{Ef} = {\bf U}(2)$. See Theorem 4.2. The first step shows that
the conclusion of Theorem 4.1 does not hold when tensoring
on an arbitrary subspace. The second step illustrates the 
theorem.  \end{example}

\begin{example} In Section 7 we consider the polynomial map defined by
$$ f(z_1,z_2) = (z_1^3, \sqrt{3} z_1z_2, z_2^3). $$
We show that $\Gamma_f$ is the subgroup of ${\bf U}(2)$ generated by the diagonal unitary matrices
and the permutation matrix that switches the variables. Here $f = f_2 \oplus f_3$.
Let $A$ be the span of $(0,1,0)$.  Then $Ef= z^{\otimes 3}$ and
$\Gamma_{Ef} = {\bf U}(2)$. Thus $\Gamma_f \subsetneq \Gamma_{Ef}$. \end{example}

We next prove the following useful result regarding juxtaposition of maps.

\begin{proposition} Let $f,g$ be polynomial proper maps with the same source
and with Hermitian-invariant groups $\Gamma_f$ and $\Gamma_g$. Let $m$ be an integer larger
than the degree of $f$. Fix $\theta \in (0, {\pi \over 2})$.
Define a rational proper map $j(f,g)$ by
$$ j(f,g) = J_\theta(f, g \otimes z^{\otimes m}). $$
Then $\Gamma_{j(f,g)} \cap {\bf U}(n) = \Gamma_f \cap \Gamma_g \cap {\bf U}(n)$.
 \end{proposition}

\begin{proof} Write $c = \cos(\theta)$ and $s = \sin(\theta)$.  We note that
$${\mathcal H}(j(f,g)) =  c^2 ||f||^2 + s^2 ||z||^{2m} ||g||^2 - 1. $$
Hence, if $\gamma$ is unitary, then 
$${\mathcal H}(j(f,g) \circ \gamma) =  c^2 ||f\circ \gamma ||^2 + s^2 ||z||^{2m} ||g\circ \gamma||^2 - 1. \eqno (26) $$
If $||f \circ \gamma||^2 = ||f||^2$ and $||g \circ \gamma||^2 = ||g||^2$, then (26) gives
$|| j(f,g) \circ \gamma||^2 = ||j(f,g)||^2$. Therefore, 
$\Gamma_f \cap \Gamma_g \cap {\bf U}(n) \subseteq \Gamma_{j(f,g)} \cap {\bf U}(n)$.

To establish the opposite containment, choose $\gamma \in \Gamma_{j(f,g)} \cap {\bf U}(n)$.
It follows that $|| j(f,g) \circ \gamma||^2 = ||j(f,g)||^2$. 
Put $g_m= g \otimes z^{\otimes m}$.
Since the degree of $f$ is smaller than the degree of $g_m$, and $\gamma$ is linear, thus preserving degree, this equality forces
both $||f \circ \gamma||^2 = ||f||^2$ and $||g_m \circ \gamma||^2 = ||g_m||^2 = ||z||^{2m} ||g||^2$. 
The conclusion follows. \end{proof}

\begin{corollary} Assume in Proposition 4.1 that $f$ is of degree at least $2$
and that $f(0)=0$. 
Then $\Gamma_{j(f,g)} = \Gamma_f \cap \Gamma_g$.\end{corollary}
\begin{proof} By Theorem 3.2 or Corollary 3.3, both $\Gamma_f$ and $\Gamma_{j(f,g)}$ lie in ${\bf U}(n)$. \end{proof}

\begin{remark} Both statements $\Gamma_{J(f,g)} = \Gamma_f \cap \Gamma_g$
and $\Gamma_{J(f,g)} \cap {\bf U}(n) = \Gamma_f \cap \Gamma_g \cap {\bf U}(n)$ are false in general.
We need the tensor power $z^{\otimes m}$ to prevent interaction between $f$ and $g$.
See the next example.
\end{remark}

\begin{example} Put $f(z) = (z_1, z_1z_2, z_2^2)$ and $g(z) = (z_1^2, z_1z_2, z_2)$.
As in Example 4.5, $\Gamma_f = \Gamma_g = {\bf U}(1) \oplus {\bf U}(1)$.
With $\theta = {\pi \over 4}$, we have (after a unitary map)
$$ J_\theta(f,g) = {\sqrt{2} \over 2} (z_1^2, \sqrt{2} z_1z_2, z_2^2) \oplus {\sqrt 2 \over 2}(z_1,z_2). $$
Then $\Gamma_{J_\theta(f,g)} = {\bf U}(2) \ne {\bf U}(1) \oplus {\bf U}(1)$.
\end{example}

Following [D3], we next compute ${\mathcal H}(f)$ when $f$ is the tensor product
of automorphisms. For $||a|| < 1$ we put $c_a = 1-||a||^2$. We put 
$\omega_j = \omega_j(z,{\overline z})  = |1 - \langle z,a_j\rangle|^2$.

\begin{proposition} Let $f = {p \over q}$ be the tensor product
of $K$ automorphisms $\phi_{a_j}$. Assume each $a_j \ne 0$.
Formula (27) holds for the Hermitian form ${\mathcal H}(f)$:

$$ {\mathcal H}(f) = ||p||^2 - |q|^2 = \prod_{j=1}^K (c_j \rho + \omega_j)- \prod_{j=1}^K \omega_j. \eqno (27) $$
\end{proposition}

\begin{proof} By (14), applied when $f$ is the identity map, the squared 
norm of the numerator of $\phi_{a_j}$ can be written (where $c_j = 1 - ||a_j||^2$):
$$ c_{j} \rho + \omega_j.  $$
Note that the squared norm of a tensor product is the product of the squared norms of the factors.
Hence the numerator of $f$ is the tensor product of the corresponding numerators and the denominator is the product of the corresponding denominators. 
As claimed, we obtain
$$ ||p||^2 - |q|^2 = \prod_{j=1}^K (c_j \rho + \omega_j)- \prod_{j=1}^K \omega_j.  $$\end{proof}

Formula (27) is a polynomial $\sum_{j=1}^K B_j \rho^j$ in the defining function $\rho$. The
coefficients $B_j$ are functions of $z,{\overline z}$ and symmetric in the $\phi_j$.
They satisfy simple formulas such as
$$ B_0 = 0 $$ 
$$ B_1 = \sum_j c_j \prod_{k \ne j} \omega_k$$
$$ B_2 = \sum_{j \ne k} c_j c_k \prod_{l \ne j,k} \omega_l$$
$$ B_K = \prod_{j=1}^K c_j. $$
These formulas indicate the symmetry in the points $a_j$. 

The next lemma enables us to go a bit beyond Corollary 3.3. 

\begin{lemma} Assume $|\lambda_1|^2 + |\lambda_2|^2 =1$, with $\lambda_2 \ne 0$.
Let $h$ be a proper mapping between balls.
Define a proper map  by $f = \lambda_1 \oplus \lambda_2 h$.
Then $\Gamma_f = \Gamma_h$. \end{lemma}
\begin{proof} Let $a= \lambda_1 \oplus 0$. Let $\psi_a$ be an automorphism
of the target with $\psi_a(0)=a$. Put $g= \psi_a \circ f$. Proposition 3.2 implies that
$\Gamma_g = \Gamma_f$. A simple computation shows that $g=0 \oplus(-h)$
and the conclusion follows.  \end{proof}

We gather much of the information we have obtained in the following theorem.

\begin{theorem} Let $\lambda = (\lambda_1,..., \lambda_k)$ satisfy $||\lambda||^2 =1 $.
Let $(m_1,...,m_k)$ be a $k$-tuple of distinct non-negative integers.
Let $f$ be the proper map defined by 
$$ f(z) = J_\lambda(z^{\otimes m_1}, ..., z^{\otimes m_k}) = \sum_{\oplus} \lambda_j z^{\otimes m_j}. $$ 
\begin{itemize}
\item Suppose $f(z) = z$. Then $\Gamma_f = {\rm Aut}({\mathbb B}^n)$.
\item For $m\ge 2$, put $f(z)=z^{\otimes m}$. Then $\Gamma_f = {\bf U}(n)$. 
\item If some $m_j \ge 2$, then $\Gamma_f = {\bf U}(n)$. 
\item If $k=2$ and $(m_1,m_2)= (0,1)$, then $\Gamma_f = {\rm Aut}({\mathbb B}^n)$. 
\item In all these cases, the source rank satisfies ${\bf s}(f) =1$. 
\end{itemize}
\end{theorem}

\begin{proof} By Definition 1.3,
$$ {\bf S}({\bf U}(n)) \le n- (n-1) =1. $$
The last item therefore follows from the previous items
because $\Gamma_f$ contains ${\bf U}(n)$.

The first and fourth items are clear.
The second item follows from Corollary 3.3. When none of the exponents $m_j$ is zero, the third
statement also follows from Corollary 3.3; if however some $k_l=0$ and some $k_l \ge 2$,
then Lemma 4.1 and Corollary 3.3 combine to yield the third item. \end{proof}

We next provide an alternative proof of Theorem 4.2 to illustrate
the connection with Hermitian forms more directly.

{\bf Alternate proof of Theorem 4.2}. Let $\rho=||z||^2-1$.
If $f(z) = z^{\otimes m}$, then ${\mathcal H}(f) = (1+ \rho)^m -1$. 
It follows that ${\bf U}(n) \subseteq \Gamma_f$. 
When $m=1$, the map $f$ is the identity and its invariant group is the full automorphism group.

When $m \ge 2$, we show directly that $\Gamma_f = {\bf U}(n)$. 
Let $\psi$ be an automorphism of the source ball.
Assume for some positive constant $c$ that
$${\mathcal H}(f \circ \psi) = c {\mathcal H}(f)$$
and hence that
$$ ||(U(a - L_az))^{\otimes m}  ||^2 - |1 - \langle z,a \rangle|^2 = c \left( ||z||^{2m} -1\right). \eqno (28) $$
To establish the result we need to show that $ \psi$
is necessarily unitary, or equivalently, that $a=0$.
First we rewrite the first term of the left-hand side of (28) as
$$   ||a-L_az||^{2m}.  $$
Using the technique of Proposition 4.2, we replace
this term with  
$$( c_a \rho + |1- \langle z,a \rangle|^2)^m $$ and obtain the crucial identity
$$ \left( c_a \rho + |1- \langle z,a \rangle|^2\right)^m - |1- \langle z,a \rangle|^{2m} = c( ||z||^{2m} - 1) = 
 c((1+\rho)^m -1). \eqno (29) $$
This identity holds for all $z$.
We expand the left-hand side of (29) as a polynomial in $\rho$, and then divide both sides by $\rho$.
After dividing, we set $\rho =0$ and obtain
$$ m c_a |1-\langle z,a \rangle|^{2m-2} = c m. \eqno (30) $$
The right-hand side of (30) is a constant; 
thus the left-hand side is also independent of $z$ (now restricted
to the unit sphere). Hence $a=0$, and $\psi$ is unitary.
It follows that $\Gamma_f \subseteq {\bf U}(n) \subseteq \Gamma_f$. 

Now consider the orthogonal sum of tensor powers, where some $m_j \ge 2$. 
We wish to show that $\Gamma_f = {\bf U}(n)$. Since ${\mathcal H}(f)$ is a function of $\rho$,
we have ${\bf U}(n) \subseteq \Gamma_f$ as before. We must show the opposite containment. 
Assume that there are two summands; the general case is similar. 
As before $||z||^{2m} = (1+ \rho)^m$ and we compare the coefficients of the first power
of $\rho$ that appears on both sides of 
$$ {\mathcal H}(f \circ \gamma) = k {\mathcal H}(f). $$
We have 
$$ {\mathcal H}(f) = |\lambda_2|^2 ||z||^{2m_2} + |\lambda_1|^2 ||z||^{2m_1} -1=  
|\lambda_2|^2 (1+ \rho)^{m_2} + |\lambda_1|^2  (1+ \rho)^{m_1} -1. $$
The coefficient of the linear term in $\rho$ in $k{\mathcal H}(f)$ is the constant
$$ k(m_1 |\lambda_1|^2 + m_2 |\lambda_2|^2). $$
Next we find the coefficient of $\rho$
in the expression for ${\mathcal H}(f \circ \gamma)$.
By a calculation similar to (28), it contains the factor $|1- \langle z,a \rangle|^{2m_2}$. These coefficients must be equal for all
$z$ on the unit sphere. This equality can happen only if $\langle z, a\rangle = 0$ for
all $z$ on the sphere. Hence $a=0$, and $\gamma$ is unitary. Thus $\Gamma_f \subseteq {\bf U}(n) \subseteq \Gamma_f$. 
Similar reasoning proves the result for an arbitrary (finite) number of summands.

\begin{remark}Orthogonal sums of tensor powers can exhibit subtle behavior.
Put
$$ f(z) = {{\sqrt 2} \over 2} \ (z \oplus z^{\otimes 2}). $$ 
There are automorphisms $\psi$ and $\gamma$ that move the origin
but for which $g = \psi \circ f \circ \gamma$ is a polynomial.
In this case, $\Gamma_g = \gamma^{-1} \circ \Gamma_f \circ \gamma$ is not 
contained in ${\bf U}(n)$. 
\end{remark}

\section{Group invariant rational maps}

We prove two theorems in this section of the following sort.
We assume a property of $\Gamma_f$ and put $f$ into a normal form.
We start with the following lemmas.

\begin{lemma} Assume $g:{\mathbb C}^n \to {\mathbb C}$ is a holomorphic polynomial
with $g(0) \ne 0$ and $|g(z)|^2 = \Phi(|z_1|^2, ..., |z_n|^2)$ for some real-analytic function
$\Phi$ of $n$ real variables. Then $g$ is a constant.  \end{lemma}
\begin{proof} Write $g(z) = g(0) + h(z)$ where $h$ vanishes at $0$.
Then 
$$ |g(z)|^2 = |g(0) + h(z)|^2 = |g(0)|^2 + |h(z)|^2 + 2 {\rm Re} ( h(z) {\overline {g(0)}}). $$
If $h$ is not identically $0$, then the Taylor expansion of $|g(z)|^2$ includes holomorphic terms
and hence cannot be a function $\Phi$ of the squared absolute values 
of the coordinate functions. \end{proof}

Note that the conclusion is false
without the assumption $g(0)\ne 0$. For example, we could put $g(z)= z_1$. 

Let $I_m$ denote the $m-$dimensional identity matrix, including the possibility that $m = 0$. For any fixed triple $m,k,l$ of non-negative integers with $m+k+l=n$, we write the coordinates in 
$\mathbb{C}^n$ as 
$$z=(z', z'', z''')= z' \oplus z'' \oplus z'''.$$ Here
$z'$ has $m$ components, $z''$ has $k$ components, and $z'''$ has
$l$ components.

\begin{lemma} Let $g(\zeta, {\overline \zeta})$ be a Hermitian
symmetric polynomial on ${\mathbb C}^k$.
Assume $g(U\zeta, {\overline {U \zeta}}) = g(\zeta, {\overline \zeta})$ for every $U \in {\bf U}(k)$.
Then there is a polynomial $h$ such that
$$ g(\zeta, {\overline \zeta}) = h(||\zeta||^2). $$
A similar conclusion holds when $g$ depends on other parameters; $h$ will depend on these other parameters
as well.
\end{lemma}

\begin{proof} We prove the lemma by induction on the degree of $g$. 
The conclusion is clear when $\mathrm{deg}(g) \leq 1$.
Now assume that the conclusion holds when $\mathrm{deg}(g) \leq \mu$ for some $\mu \geq 1$ 
and suppose $\mathrm{deg}(g) =\mu +1$. 
Fix a point $p$ on the unit sphere.
The assumption implies 
$$ g(\zeta, {\overline \zeta}) = g(p, \overline{p}) $$ 
for each $\zeta$ on the unit sphere. Consider the quotient
$$ u(\zeta, {\overline \zeta}) = {g(\zeta, {\overline \zeta}) - g(p,{\overline p}) \over ||\zeta||^2 -1 }. $$
The quotient $u$ satisfies the same hypotheses as $g$, but it is of lower degree.
By the induction hypothesis we can write $u(\zeta, {\overline \zeta}) = H(||\zeta||^2)$ for some $H$.
Thus 
$$g(\zeta, {\overline \zeta}) = (||\zeta||^2 -1) H(||\zeta||^2) + g(p,{\overline p}) = h(||\zeta||^2). $$ \end{proof}

\begin{lemma} Let $\Gamma_f$ be the Hermitian invariant group of a proper rational map $f$.
Suppose for non-negative integers $m,k,l$ that
$\{I_{m}\} \oplus {\bf U}(k) \oplus \{ I_{l}\}$ is contained in a conjugate of $\Gamma_f$ by $\varphi$. Then there is real-valued polynomial
$h$ such that 
$$||f \circ \varphi||_l^2=h(z', \overline{z'}, ||z''||^2, z''', \overline{z'''}). $$ 
\end{lemma}
\begin{proof}
It suffices to prove the result when $\varphi$ is the identity, in which case we 
assume $\{I_{m}\} \oplus {\bf U}(k) \oplus \{ I_{l}\} \subseteq \Gamma_f$. We write 
$\mathcal{H}(f)=g(z',\overline{z'}, z'', \overline{z''}, z''',\overline{z'''})$. We obtain
$$ g(z', \overline{z'}, z'', \overline{z''}, z''', \overline{z'''})=g(z',\overline{z'}, Uz'', \overline{Uz''}, z''', \overline{z'''}) $$
for every $U \in {\bf U}(k)$. The conclusion follows from Lemma 5.1, with $z''=\zeta$. 
\end{proof}

\begin{lemma}
Let $h$ be a real-valued polynomial in $\mathbb{C}^m \times \mathbb{R} 
\times \mathbb{C}^l$.  There exist linearly independent holomorphic 
polynomials $f_1, 
\dots f_p, g_1, \dots, g_q,$  such that
$$h(z', \overline{z'}, ||z''||^2, z''',\overline{z'''})=
\sum_{i=1}^p |f_i(z)|^2-\sum_{j=1}^q |g_j(z)|^2. $$
Moreover, there are 
real-valued polynomials $\Psi$ and $\Phi$, defined  
in $\mathbb{C}^m \times \mathbb{R} \times \mathbb{C}^l$, with
$$||f(z)||^2 = \sum_{i=1}^p|f_i(z)|^2=\Psi(z',\overline{z'},
|z''|^2, z''', \overline{z'''})$$
and
$$ ||g(z)||^2 = \sum_{j=1}^q |g_j(z)|^2=\Phi(z',\overline{z'},|z''|^2, 
z''', \overline{z'''}).$$ 
\end{lemma}
\begin{proof} Write $h(z', \overline{z'}, ||z''||^2, z''', \overline{z'''})=
\sum_{k=0}^d C_k(z', \overline{z'},  z''', \overline{z'''})||z''||^{2k}$. 
Here the $C_k$ are real-valued polynomials. For each $1 \leq k \leq d$ we write
$$C_k(z',\overline{z'},  z''', \overline{z'''})=\sum_{i=1}^{p_k} 
|a_{ki}(z)|^2 -\sum_{j=1}^{q_k} |b_{kj}(z)|^2 = ||a_k(z)||^2 - 
||b_k(z)||^2. 
$$
Here $\{a_{k1}, \cdots, a_{kp_{k}}, b_{k1}, \cdots, b_{kq_k}\}$ is 
a linearly independent set. Then 
$$h(z', \overline{z'}, ||z'||^2, z''', \overline{z'''})=\sum_{k=0}^d \left( \sum_{i=1}^{p_k} |a_{ki}(z)|^2 -\sum_{j=1}^{q_k} |b_{kj}(z)|^2 \right) ||z''||^{2k} $$
$$ = \sum_k ||a_k(z) \otimes (z'')^{\otimes k}||^2 - ||b_k(z) \otimes (z'')^{\otimes k}||^2 = ||f(z)||^2 - ||g(z)||^2. $$
Linear independence is preserved under tensoring and the existence of the desired $f_j$ and $g_j$ follows.
\end{proof}

We now can prove two useful results. 

\begin{theorem} Assume $N \ge n\ge 1$. 
Let $f: \mathbb{B}^n \to \mathbb{B}^N$ be a rational holomorphic proper map with Hermitian invariant group
$\Gamma_f$. Then $\Gamma_f$ 
contains an $n$-torus, that is,  a conjugate of ${\bf U}(1) \oplus ... \oplus {\bf U}(1)$,
if and only if $f$ is spherically equivalent to a monomial map.
\end{theorem}

\begin{proof} If $f$ is equivalent to a monomial map, then $||f||^2$ is invariant under the torus
${\bf U}(1) \oplus ... \oplus {\bf U}(1)$.
We will prove the converse assertion.  Assume ${\bf U}(1) \oplus ... \oplus 
{\bf U}(1)$ is
a subset of $\varphi^{-1} \circ \Gamma_f \circ \varphi$ for some $\varphi \in \mathrm{Aut}(\mathbb{B}^n)$.
Then, replacing $f$ by $f \circ \varphi$ if necessary,  
we can assume ${\bf U}(1) \oplus ... \oplus {\bf U}(1)$ is a subset of 
$\Gamma_f$. 

By composing $f$ with an automorphism in the target, we can assume $f(0)=0$.
Write $f=\frac{P}{q}$; as usual we assume $P(0)=0$ and $q(0)=1$ and that ${P \over q}$
is reduced to lowest terms. 
By Lemma 5.2,  $\mathcal{H}(f)$ is a polynomial in the variables $|z_1|^2,...,|z_n|^2$.
By Lemma 5.4, we have 
$$\mathcal{H}(f)=||P||^2-|q|^2=\sum_{i=1}^p |h_i(z)|^2-|g(z)|^2 = ||h(z)||^2 - |g(z)|^2. $$
Here $h_1, ...,h_p, g$ are linearly independent holomorphic polynomials.
Moreover, $|g(z)|^2=\Phi(|z_1|^2, ...,|z_n|^2)$ and $||h(z)||^2 = \Psi(|z_1|^2,...,|z_n|^2) $ for some real polynomials
$\Phi, \Psi$ in $\mathbb{R}^n$. 

Since $g(0) \neq 0$, Lemma 5.1 guarantees that $|g|^2$ is a nonzero 
constant $c^2$. We may assume $c>0$. 
We further write
$$||h(z)||^2 = 
\sum_{i=1}^p |h_i(z)|^2=\lambda_1^2 |z^{\alpha^1}|^2+ \cdots +\lambda_{k}^2 |z^{\alpha^k}|^2$$
for distinct multi-indices $\alpha_1,...,\alpha_k$ and positive constants $\lambda_1,...,\lambda_k$. By the linear independence of $h_1, \cdots, h_p, g,$
we have  $|\alpha_i| \geq 1$ for all $1 \leq i \leq k$. Hence $h(0)=0$.
Since 
$$||h(0)||^2 - |g(0)|^2 = ||p(0)||^2 -|q(0)|^2 = -1 ,$$ the constant $c$ must equal $1$.
Then $f$ is spherically equivalent to the monomial map
$(\lambda_1z^{\alpha_1}, ..., \lambda_k z^{\alpha_k}) = J_\lambda\left(z^{\alpha_1},..., z^{\alpha_k}\right)$.
\end{proof} 

The next result tells us when $f$ is an orthogonal sum of tensor products.
In this result, as in the previous theorem, after composition with a diagonal unitary map, we may 
assume that the $\lambda_j$ are positive.  

\begin{theorem} Assume $N \ge n\ge 1$. 
Let $f: \mathbb{B}^n \to \mathbb{B}^N$ be a rational holomorphic proper map with Hermitian invariant group
$\Gamma_f$. Then $\Gamma_f$ contains a maximal compact subgroup, that is, a conjugate of 
${\bf U}(n)$, if and only if $f$ is spherically equivalent to an orthogonal sum of tensor products:
$$ f(z) = J_{\lambda}(z^{\otimes m_1}, ...,z^{\otimes m_k})\oplus 0  = \lambda_1 z^{\otimes m_1} \oplus \cdots \oplus \lambda_k z^{\otimes m_k} \oplus 0.$$
Here $k \geq 1$ and $m_k> \cdots > m_2 > m_1 \ge 0$. 
\end{theorem}
\begin{proof} The ``if'' implication follows from Theorem 4.2.
We will prove the converse. The proof is similar to
the proof of Theorem 5.1. Assume that ${\bf U}(n)$ is a subset of $\varphi^{-1} \circ \Gamma_f \circ \varphi$ 
for some $\varphi \in \mathrm{Aut}(\mathbb{B}^n)$.  Replacing $f$ by $f \circ \varphi$, we assume ${\bf U}(n)$ is a subset of $\Gamma_f$.  Again we may assume $f(0)=0$, that
$f=\frac{P}{q}$, the fraction is reduced to lowest terms,
$P(0)=0$, and $q(0)=1$.    By Lemma 5.2, 
$\mathcal{H}(f)$ is a polynomial function of $||z||^2$. By Lemma 5.4, we can write
$$\mathcal{H}(f)=||P||^2-|q|^2=\sum_{i=1}^p |h_i(z)|^2-|g(z)|^2, $$
where $h_1, ...,h_p, g$ are linearly independent holomorphic polynomials. 
Moreover, $|g|^2$ and $\sum_{i=1}^p |h_i(z)|^2 $ are also polynomials in $||z||^2$. Note that $g(0)\ne 0$. As 
above, by Lemma 5.1,  we conclude that $|g|^2$ is a nonzero constant $c^2$.
We then write
$$\sum_{i=1}^p |h_i(z)|^2=\lambda_1^2 ||z||^{2m_1}+ \cdots +\lambda_{k}^2 ||z||^{2m_k},$$
for  integers $m_1, \cdots, m_k$ and positive constants $\lambda_1,...,\lambda_k$. 
Put $\lambda = (\lambda_1,...,\lambda_k)$. 
As in the proof of Theorem 5.1, for $1 \leq j \leq k$ we have $m_j \ge 1$ and again $c=1$.
Then $f$ is spherically equivalent to 
$$ J_{\lambda}(z^{\otimes m_1}, ...,z^{\otimes m_k})\oplus 0$$
and Theorem 5.2 follows. 
\end{proof}

\section{Finite Hermitian-invariant groups}

Let $G$ be a finite subgroup of ${\rm Aut}({\mathbb B}^n)$. The main result of this section, Theorem 6.2,
constructs a rational proper map $f$ with $\Gamma_f = G$. In addition, if $G \leqslant {\bf U}(n)$, then $f$ may be chosen to be a polynomial. The proof relies on knowing a basis for the algebra of polynomials invariant under $G$.
In Theorem 6.3 we prove a slightly weaker version using an easier construction.

By Cayley's theorem each finite group $G$ is a subgroup
of the permutation group $S_n$ on $n$ letters for some $n$. We may represent $S_n$
as a subgroup of ${\bf U}(n)$ by fixing a coordinate system on ${\mathbb C}^n$ 
and considering the group of unitary maps that permute the coordinates.
We identify $G$ with a subgroup of ${\bf U}(n)$ in this way. In Theorem 6.3 we 
construct a polynomial proper map $f$ with $\Gamma_f = G$. We provide two
proofs of Proposition 6.1, the special case of Theorem 6.3 when $\Gamma_f$
is the symmetric group. In particular, we can find a polynomial map of degree $3$
whose group is $S_n$, but we note in Remark 6.1 that doing so is impossible for degree $2$.

The proofs of Theorem 6.2 and 6.3 use the results of the previous sections
and also the following result, which holds even when the target is a generalized ball.
By Corollary 3.3, when the target is the unit ball and $d > 1$, the constant
monomial must also occur in the conclusion of Theorem 6.1.

\begin{theorem}Let $f:{\mathbb B}^n \to {\mathbb B}^N_l$ be a proper polynomial map of degree $d$.
Assume that $\Gamma_f$ contains an automorphism that moves the origin.
Then there is a coordinate function $z_j$ such that each monomial $z_j^k$, for $1 \le k \le d$, 
appears in $f$. \end{theorem}

\begin{proof}
By assumption, there is some $\gamma \in \Gamma_f$ such that $\gamma \not \in {\bf U}(n)$ 
and there exists $\psi_{\gamma} \in \mathrm{Aut}(\mathbb{B}_l^N)$ such that
$$f \circ \gamma=\psi_{\gamma} \circ f.$$
We embed $\mathbb{B}^n$ as an open subset of $\mathbb{P}^n$ by  the map
$$z \rightarrow [z, 1].$$
Here $[z, s]$ are homogeneous coordinates of $\mathbb{P}^n.$ Recall 
that $\mathbb{B}_l^N$ is an open subset of $\mathbb{P}^N$. We regard 
$\gamma, 
\psi_{\gamma}$ as elements in ${\bf SU}(n, 1)$ and ${\bf SU}(m-l, l+1)$. 
Write
$$\gamma=U=(u_1, \cdots, u_{n+1})\in {\bf SU}(n,1),$$
where each $u_i$ is an $(n+1)-$dimensional column vector. 
Note $\gamma=U \in {\bf U}(n)$ if and only if 
$u_{n+1}=(0, \cdots, 0, 1)^t$. 
Write $u_{n+1}=(\lambda_1, \cdots, \lambda_{n+1})^t$. 
Since $U \in {\bf SU}(n, 1)$, we 
have $\sum_{i=1}^n |\lambda_i|^2 -|\lambda_{n+1}|^2=-1$.
Thus $\lambda_{n+1} \neq 0$. Since $\gamma \not \in {\bf U}(n)$ it follows
that $\lambda_{n+1} \neq 1$ and $\lambda_{i_0} \neq 0$ 
for some $i_0$ with $1 \leq i_0 \leq n$.
In  homogeneous coordinates, $f$ can be written as 
$  [P(z, s), s^d]$, 
where $P(z, s)=s^d f(\frac{z}{s})$. Moreover, $f \circ \gamma$ can be written as
$$[P((z, s) U), ((z, s)u_{n+1})^d].$$
Write $\psi_\gamma = V \in {\bf SU}(m-l,l+1)$. As above, $\psi_{\gamma} \circ f$ can be written as 
$$[P(z, s), s^d] V.$$
Since $f \circ \gamma=\psi_{\gamma} \circ f$, there is a nonzero constant $c$ such that
$$\left( P((z, s) U), ((z, s)u_{n+1})^d \right)=c \left( P(z, s), s^d \right) V. $$
Equating the last component in this  equation yields
$$ 
((z, s) u_{n+1})^d = c (P(z, s), s^d) v_{N+1}.
\eqno (31) $$ 
Recall that $u_{n+1}=(\lambda_1, \cdots, \lambda_{n+1})^t$. We thus have 
$$((z, s) u_{n+1})^d = \left( \cdots + \lambda_{i_0} z_{i_0}+ \cdots + \lambda_{n+1} s \right)^d. $$
Each term $z_{i_0}^k s^{d-k}$ for $ 1 \leq k \leq d $ appears on the left-hand 
side of (31) and hence such terms also appear on the right-hand 
side  of (31) and thus in $P(z, s)$. By the definition of 
$P(z, s)$, each term $z_{i_0}^k$ for $ 1 \leq k \leq d$ therefore appears in 
$f(z)$.
\end{proof}

\begin{corollary} Consider the Whitney map $W$ defined by 
$$ W(z_1,...,z_n) = (z_1,...,z_{n-1}, z_1z_n, ..., z_{n-1}z_n, z_n^2). $$
Then $\Gamma_W = {\bf U}(n-1) \oplus {\bf U}(1)$. 
In particular, the source rank satisfies ${\bf s}(W) = 2$. \end{corollary}
\begin{proof} The theorem implies that $\Gamma_W$ contains no automorphism that moves the origin.
The Hermitian norm is obviously invariant under ${\bf U}(n-1) \oplus {\bf U}(1)$. 
Thus ${\bf s}(W) \le n-(n-2)=2$ and, by Theorem 5.2, ${\bf s}(W) \ne 1$.
\end{proof}

We next prove several lemmas used in the proofs of Theorem 6.2 and 6.3.

\begin{lemma}
Let $p(z):\mathbb{C}^n \to \mathbb{C}^K$  be a polynomial of degree $d$.
There is an $\epsilon > 0$ and a polynomial map $q$ such
that $\epsilon p \oplus q$ is  a proper map that is unitarily equivalent to an orthogonal sum of tensor products $J_{\lambda}(z^{\otimes m_1}, ..., z^{\otimes m_k})$. Here the $m_j$ are distinct and $0 \leq m_j \leq d$ for each $j$.
The target dimension of $q$ can be chosen to depend only on $n$ and $d$. Equivalently,
$$ \epsilon^2 ||p(z)||^2 + ||q(z)||^2 =  \sum_j |\lambda_j|^2 ||z||^{2m_j}. \eqno (32) $$
\end{lemma}

\begin{proof}
Consider the polynomial $r(z,{\overline z}) = \sum_{j=0}^d |\lambda_j|^2 ||z||^{2j}$, where $||\lambda||^2=1$
and each $\lambda_j \ne 0$. 
Its underlying Hermitian form is positive definite. Hence, for sufficiently small $\epsilon$,
the polynomial $ r(z,{\overline z}) - \epsilon^2 ||p(z)||^2$ also has a positive definite Hermitian form, and hence is a squared norm
$||q(z)||^2$. Note: when
$p$ omits all monomials of degree $b$, we can omit the term $||z||^{2b}$ from our definition of $r$,
obtain a positive semi-definite form, and draw the analogous conclusion. 
\end{proof}

\begin{lemma} Assume that $U \in {\bf U}(n)$ and
$ \big| \prod_{j=1}^n \left(1 + U(z_j)\right)\big|^2  = \big|\prod_{j=1}^n \left(1+ z_j\right)\big|^2$.
Then $U$ is a permutation of the coordinates.
\end{lemma}

\begin{proof} The hypothesis implies that 
$ \prod_{j=1}^n \left(1 + U(z_j)\right) = e^{i\theta} \prod_{j=1}^n \left(1+ z_j\right)$ for some $e^{i\theta}$.
Evaluating at $z=0$ forces $e^{i \theta} =1$. Hence
$$ \prod_{j=1}^n \left(1 + U(z_j)\right) = \prod_{j=1}^n \left(1+ z_j\right).  $$
Expand the products and equate coefficients. It follows for each symmetric polynomial
$\sigma_k$ that $\sigma_k(Uz)= \sigma_k(z)$. Since these symmetric
polynomials precisely generate the algebra of symmetric polynomials, it follows that
$U$ itself is a permutation. \end{proof}

We also have the following lemma when we assume less about $U$.

\begin{lemma} For $z \in {\mathbb C}^n$, define $p$ by $p(z)= (..., z_j z_k, ....)$, where $1 \le j < k \le n$.
Suppose that $U$ is unitary, and $||p(Uz)||^2 = ||p(z)||^2$.
Then there is a diagonal matrix $L$ and a permutation $\sigma$ of the coordinates such that $U=L \sigma$. 
\end{lemma}
\begin{proof} First suppose $n=2$. Let $U$ be unitary. We are given
$$ p(Uz) = (u_{11}z_1 + u_{12}z_2)(u_{21}z_1 + u_{22}z_2) $$
and 
$$||p(Uz)||^2 = ||p(z)||^2 = |z_1z_2|^2. $$
The coefficients of $|z_1|^4$ and $|z_2|^4$ both vanish, and hence we get
$$ u_{11}u_{21} = 0  $$
$$ u_{12} u_{22} = 0. $$
If $u_{11} = 0$, then $U$ unitary implies $u_{21} \ne 0$ and $u_{12} \ne 0$. The second equation yields
$u_{22} = 0$. Let $L$ be the diagonal matrix with diagonal entries $u_{12}$ and $u_{21}$, 
and let $\sigma$ permute the coordinates. The conclusion holds.
If $u_{11} \ne 0$, then $u_{12}=0$ and hence $u_{22} \ne 0$.
Now the second equation gives $u_{21}=0$, 
and $U$ is diagonal. Thus the conclusion holds with $\sigma$ the identity.

Consider $n=3$. We are given
$ p(Uz) = (a_1(z), a_2(z), a_3(z)) $
where 
$$ a_1(z) = (u_{11}z_1 + u_{12}z_2 + u_{13}z_3)(u_{21}z_1 + u_{22}z_2 + u_{23}z_3) $$
$$ a_2(z) = (u_{11}z_1 + u_{12}z_2 + u_{13}z_3)(u_{31}z_1 + u_{32}z_2 + u_{33}z_3) $$
$$ a_3(z) = (u_{21}z_1 + u_{22}z_2 + u_{23}z_3)(u_{31}z_1 + u_{32}z_2 + u_{33}z_3). $$
The coefficient of each $|z_j|^4$ in $||p(z)||^2$ is $0$. The coefficients $c_l$ of $|z_l|^4$ in 
$||p(Uz)||^2$ are given by
$$ c_1 = |u_{11}u_{21}|^2 + |u_{11} u_{31}|^2 + |u_{21}u_{31}|^2 $$
$$ c_2 = |u_{12}u_{22}|^2 + |u_{12} u_{32}|^2 + |u_{22}u_{32}|^2 $$
$$ c_3 = |u_{13}u_{23}|^2 + |u_{13} u_{33}|^2 + |u_{23}u_{33}|^2. $$
Hence each of the nine terms separately vanishes. 

For general $n$ these equations become, for each $l$ with $1 \le l \le n$,
$$ \sum_{1 \le j < k \le n} |u_{jl} u_{kl}|^2 = 0. \eqno (33) $$
For each column of $U$, formula (33) implies that the product of any pair of distinct entries vanishes.
Hence there can be at most one non-zero element in each column. 
Since $U$ is invertible, there is exactly one non-zero element in each column. 
Since $U$ is unitary, each of these entries has modulus $1$. Let $L$ be the diagonal matrix with these entries
and $\sigma$ the appropriate permutation. The conclusion follows.
\end{proof}

We now prove what is perhaps the main result of this paper.

\begin{theorem} Let $G$ be a finite subgroup of ${\rm Aut}({\mathbb B}^n)$.
Then there is a rational proper map $f:{\mathbb B}^n \to {\mathbb B}^N$
for which $\Gamma_f = G$. When $G$ is a finite subgroup of ${\bf U}(n)$, 
we may choose $f$ to be a polynomial. \end{theorem}

\begin{proof} We first assume $G$ is a subset of ${\bf U}(n)$ 
and construct the polynomial map $f$.
Let ${\mathcal A}$ denote the algebra of $G$-invariant polynomials.
By Noether's theorem (see [Sm]), ${\mathcal A}$ is finitely generated.
Let $\{1,h_1,...,h_K\}$ denote a basis for ${\mathcal A}$; assume $h_i(0)=0$ for each $i$.
Put $h= (h_1,...,h_K)$; thus $h:{\mathbb C}^n \to {\mathbb C}^K$. Then $h$
is precisely $G$-invariant; that is, for $\gamma \in {\bf U}(n)$, we have $h \circ \gamma = h$ 
if and only if $\gamma \in G$.

After tensoring each summand with an appropriate tensor power $z^{\otimes m_j}$ we consider the map
$$  p =  \big((1+ h_1) \otimes z^{\otimes m_1}\big) \oplus\big((1+ h_2) \otimes z^{\otimes m_2}\big) \oplus ... \oplus  \big( (1+ h_K) \otimes z^{\otimes m_K}\big). $$
The $m_j$ are chosen in order to guarantee that all the monomials in each summand are distinct from
those in the other summands. By also assuming each $m_j \ge 1$ we obtain $p(0)=0$.
By Lemma 6.1, for some $\epsilon > 0$, we can find a polynomial $q$ such 
$ \epsilon p \oplus q$ is a proper polynomial map such that (32) holds.
Put $f = \epsilon p \oplus (q \otimes z^{\otimes m})$, where $m \ge {\rm deg}(p) +1$.
By Corollary 3.3, $\Gamma_f \subseteq {\bf U}(n)$.

{\bf Claim}: $\Gamma_f = G$.

Let $\gamma \in {\bf U}(n)$. Then (32) implies  $||p \circ \gamma||^2 = ||p||^2$ if and only if
$ ||q \circ \gamma||^2 = ||q||^2$. Hence, if $\gamma \in G$, then $\gamma$ preserves both $||p||^2$ and $||q||^2$.
Therefore $ G \subseteq \Gamma_f$. To prove the opposite inclusion, let $\gamma \in \Gamma_f \subseteq {\bf U}(n)$.
Then $||f \circ \gamma||^2 = ||f||^2$. Since $\gamma$ preserves the degrees of polynomials, we have
$ ||p \circ \gamma||^2 = ||p||^2$. Therefore, $\gamma$ preserves 
$$||(1+h_i) \otimes z^{\otimes m_i}||^2 = |1+h_i|^2 \ ||z||^{2m_i}$$ 
for each $i$. Since $\gamma$ is unitary, it preserves $||z||^2$ and thus also preserves 
$$|1+h_i|^2 = 1 + h_i + {\overline {h_i}} + |h_i|^2. $$
Hence $\gamma$ must preserve each pure term $h_i$, and therefore $\gamma \in G$, establishing the claim and proving
the theorem when $G \subseteq {\bf U}(n)$.

Next we assume $G$ is an arbitrary finite subgroup of ${\rm Aut}({\mathbb B}^n)$. By Lie group theory (see [HT]),
$G$ is contained in a conjugate of ${\bf U}(n)$. For some $\chi \in {\rm Aut}({\mathbb B}^n)$, we thus  have
$G_0 = \chi \circ G \circ \chi^{-1} \subseteq {\bf U}(n)$. By the result proved above for ${\bf U}(n)$,
there is a polynomial proper map $f$ for which $\Gamma_f = G_0$. By Proposition 3.2,
$\Gamma_g = G$ when $g = f \circ \chi$.
\end{proof}

In Theorem 6.3 we prove a slightly weaker statement but with a more constructive proof.
In this version, we assume that the group
is represented as a subgroup of the unitary group as indicated. At the end of the section
we compare the two proofs. We also provide two proofs of the following related proposition.

\begin{proposition} For each $n$, there is a polynomial proper map $f:{\mathbb B}^n \to {\mathbb B}^N$ 
for which $\Gamma_f$ is the symmetric group $S_n$. \end{proposition}
\begin{proof} Corollary 6.2 gives an example where $\Gamma_f$ is trivial when $n=1$.
In the first proof we assume $n\ge 2$. Begin with
the map $z \mapsto p(z) = (..., z_jz_k, ...)$ whose components are all the quadratic monomials
$z_jz_k$ with $1 \le j < k \le n$. As in the proof of Lemma 6.1, we can find a quadratic polynomial map
$$ \xi(z)=(z_1^2, \cdots, z_n^2),$$
such that
$$ 2 ||p(z)||^2 + ||\xi(z)||^2 = ||z||^4, $$
and hence $\sqrt{2} p \oplus \xi$ is a polynomial proper map. For any $\gamma \in U(n)$, the term $||z||^4$ is invariant under $\gamma$. It follows
that  $||p \circ \gamma||^2=||p||^2$ if and only if $||\xi \circ \gamma||^2=||\xi||^2.$  We put
$g = (\sqrt{2}p \otimes z) \oplus \xi $. Then $g$ is a polynomial proper map.
We next compute the Hermitian invariant group $\Gamma_g$.

{\bf Claim}: $\Gamma_g=\left( {\bf U}(1) \oplus ... \oplus {\bf U}(1) \right)  \times S_n.$

By Theorem 6.1, $\Gamma_g \subseteq {\bf U}(n)$. If $\gamma \in {\bf U}(n)$
and $||p \circ \gamma||^2 = ||p||^2$, then also $||\xi \circ \gamma||^2 = ||\xi||^2$;
thus $\gamma \in \Gamma_g$. Since $p$ is a monomial map
and $||p||^2$ is symmetric in the variables, $\left( {\bf U}(1) \oplus ... \oplus {\bf U}(1) \right)  \times S_n \subseteq \Gamma_g$. To verify the claim we must prove the opposite inclusion:
$$\Gamma_g \subseteq \left( {\bf U}(1) \oplus ... \oplus {\bf U}(1) \right)  \times S_n. $$

Let $\gamma \in \Gamma_g \subseteq {\bf U}(n)$. By Remark 3.1, $||g \circ \gamma||^2=||g||^2$. 
Since $\gamma$ preserves the degree of a polynomial, we have $||p \circ \gamma||^2 = ||p||^2$. 
Lemma 6.3 then implies that $\gamma = L\sigma$ for a diagonal unitary $L$ and permutation $\sigma$.
Therefore we have $\gamma \in \left( {\bf U}(1) \oplus ... \oplus {\bf U}(1) \right)  \times S_n$. 
Thus $\Gamma_g \subseteq \left( {\bf U}(1) \oplus ... \oplus {\bf U}(1) \right)  \times S_n \subseteq \Gamma_g$ and hence the claim holds.

Next we put $\alpha(z) =1 + \sum_{i=1}^n z_i$. By Lemma 6.1, there is an $\epsilon >0$ and 
a polynomial $\beta(z)$ of first degree 
such that $\epsilon \alpha \oplus \beta$ is an orthogonal sum of tensor products. 
Let $h$ be the polynomial proper map defined by $h=\epsilon \alpha \oplus (\beta \otimes z)$.
We use these maps to define a polynomial proper map $f$ as follows:

$$ f = J_{\frac{\pi}{4}}(g, h) = \frac{\sqrt{2}}{2}\big( \sqrt{2}( p \otimes z) \oplus \xi \oplus (\epsilon \alpha) \oplus (\beta \otimes z)\big). $$

We claim that $\Gamma_f=S_n$. By construction, $S_n \subseteq \Gamma_f$. 
We must prove that $\Gamma_f \subseteq S_n$. Note that $f$ is degree $3$ but contains no term $z_j^3$.
By Theorem 6.1, $\Gamma_f \subseteq {\bf U}(n)$. By Proposition 4.1 and the first claim, 
$$\Gamma_f=\Gamma_{g} \cap \Gamma_{h} \subseteq \left( {\bf U}(1) \oplus ... \oplus {\bf U}(1) \right)  \times S_n.$$ 
Let $\gamma \in \Gamma_f$. Then $||f \circ \gamma||^2=||f||^2$. Since $\gamma$ preserves degrees, the holomorphic linear terms are preserved. These terms arise only in $|\alpha|^2$.
Thus $\gamma$ must preserve $\sum_{i=1}^n z_i$, and therefore $\gamma \in S_n$.
\end{proof}

We next give a second proof of Proposition 6.1. It is easier to understand but the resulting map
is of higher degree. This proof also works when $n=1$.

\medskip

{\bf Second proof of Proposition 6.1}. \begin{proof} We first put $g(z) = \prod_{j=1}^n (1+z_j)$.
For small $\epsilon$, Lemma 6.1 implies there is a vector-valued polynomial map $h$ of degree $n$ such that
$$ |\epsilon g(z)|^2 + ||h(z)||^2 = \sum_{j=0}^n \lambda_j ||z||^{2j} \eqno (*)$$
and $\epsilon g \oplus h$ is a proper map. 
Now replace $g$ by $G=g \otimes z$ and $h$ by $H=h \otimes z^m$ where $m\ge n+2$ to get a proper polynomial map
$p$.
Proposition 4.1 implies that $\Gamma_p = \Gamma_G \cap \Gamma_H$. By Corollary 3.3, we have
$\Gamma_p \subseteq {\bf U}(n)$. Let $U \in \Gamma_f$. Since $U$ preserves degrees,
$|g \circ U|^2 = |g|^2$. By Lemma 6.3, the only unitary maps $U$ for which $|g \circ U|^2 = |g|^2$
are permutations. Hence the same holds for $G$. Therefore $\Gamma_p \subseteq S_n$.
We have $||p \circ \sigma||^2 = ||p||^2$ for each permutation $\sigma$
and therefore $S_n \subseteq \Gamma_p$.  \end{proof}

\begin{corollary} For $z \in {\mathbb C}$, put $ p(z) = {1 \over 2} \ \left(z+z^2, z^2 - z^3 \right)$.
Then $p:{\mathbb B}^1 \to {\mathbb B}^2$ is proper and $\Gamma_p$ is trivial. \end{corollary}
\begin{proof} The map $p$ is suggested by the proof.
We check the result directly. By Corollary 3.3, we have $\Gamma_p \subseteq {\bf U}(1)$.
Hence, if $\gamma \in \Gamma_f$, then $(p \circ \gamma) (z) = p(e^{i\theta} z)$. 
Using Proposition 3.3 and Remark 3.1 we see that
$$ |e^{i\theta} z+( e^{i\theta} z)^2|^2 + |(e^{i\theta}z)^2- (e^{i\theta}z)^3|^2 $$
must be independent of $e^{i\theta}$, which can happen only if $e^{i\theta}=1$.
\end{proof}

\begin{theorem} Let $G$ be a finite group, expressed as a subgroup
of the symmetric group $S_n$, which is represented as a subgroup
of ${\bf U}(n)$.  Then there is a polynomial proper holomorphic map $f:{\mathbb B}^n \to {\mathbb B}^N$
with $\Gamma_f = G$.
\end{theorem}
\begin{proof} Proposition 6.1 establishes the result when $G=S_n$. 
We next let $G \leqslant S_n$ and construct a polynomial proper
map $g$ such that $\Gamma_g=G$. Let $\mu_1< \mu_2< \cdots< \mu_n$ be $n$ distinct positive integers. Put
$$ \tau(z)=1 + \sum_{\sigma \in G} z_{\sigma(1)}^{\mu_1}\cdots z_{\sigma(n)}^{\mu_n} = 1 + t(z).$$
Here $t(z)$ is $G$-invariant and includes the term $z_1^{\mu_1} \cdots z_n^{\mu_n}$. By Lemma 6.1, 
there exists $\epsilon >0$ and a holomorphic polynomial map $q$ such that $\epsilon \tau \oplus q$ is unitarily equivalent to an orthogonal sum of tensor products. As before, for any $\gamma \in {\bf U}(n)$, we have 
$||\tau \circ \gamma||^2=||\tau||^2$ if and only if $||q \circ \gamma||^2=||q||^2$. 
Put $g_1 = \epsilon \tau \oplus q  \otimes z^{\otimes k_3}$, where $k_3 > \sum_{i=1}^n \mu_i$.
For each $\gamma \in G \leqslant S_n$ we have $||\tau \circ \gamma||^2=||\tau||^2$. Therefore
$G  \subseteq \Gamma_{g_1}$. Finally let $f$ be as in Proposition 6.1 with $\Gamma_f=S_n$.
For $k_4$ larger than the degree of $f$, put  $ g=J_{\frac{\pi}{4}}(f, g_1 \otimes z^{\otimes k_4})$.

We claim that $\Gamma_g=G$. By Theorem 6.1, $\Gamma_g \leqslant{\bf U}(n)$.  By Proposition 4.1, 
$$\Gamma_g=\Gamma_f \cap \Gamma_{g_1}.$$
We conclude that $G \subseteq \Gamma_g \subseteq S_n.$ We now prove that $\Gamma_g \subseteq G.$ 
Let $\gamma \in \Gamma_g \subseteq S_n.$ Again since $\gamma$ preserves the degree of polynomials, 
we have $||\tau \circ \gamma||^2=||\tau||^2$. But
$$ ||\tau \circ \gamma||^2=|t \circ \gamma|^2+ t \circ \gamma+ \overline{t \circ \gamma} +1$$
$$ ||\tau||^2=|t|^2+t+\overline{t}+1.$$
Hence $t \circ \gamma=t$. Since $\gamma \in S_n$, it maps each term $z_1^{\nu_1}\cdots z_n^{\nu_n}$ to some other term in $t$. Thus $\gamma \in G$ and $\Gamma_g \subseteq G$. 
We have established the claim and finished the proof of Theorem 6.3.
\end{proof}

Theorem 6.2 is decisive, but it relies on the theorem of E. Noether
that the algebra of polynomials invariant under the group is finitely generated. 
See [Sm] for considerable discussion about this result and how to bound the number of generators
and their degrees. Our proof of Theorem 6.2 increases in complexity 
when the degrees of the generating polynomials increase. Furthermore, to make the map explicit we need to know
all the generators. 

The advantage of Theorem 6.3 is that we first reduce to the case of the symmetric group, using
Proposition 6.1. In Theorem 6.3 we start with a quadratic polynomial map $p$. From it we construct a polynomial
map $g$ for which $\Gamma_g=\left( {\bf U}(1) \oplus ... \oplus {\bf U}(1) \right)  \times S_n$.
Then we can use a single additional invariant polynomial, namely $\sum z_j$, to eliminate
the factor of $\left( {\bf U}(1) \oplus ... \oplus {\bf U}(1) \right)$. By working
in the symmetric group, we cut down to the subgroup $G$ using the map $\tau$
in the proof. The polynomial $g$ found in the proof is explicit and is of degree $3$. 
The map found using the second
proof of Proposition 6.1 is perhaps more natural, but it is of higher degree. We ask the following
questions. Given $\Gamma_f$, what is the smallest source dimension for which there
is a rational map $g$ with $\Gamma_g$ isomorphic to $\Gamma_f$? Given a finite group $G$,
what is the smallest degree of a rational map $f$ with $\Gamma_f = G$?

\begin{remark} By a result from [L], a quadratic proper rational map $f$ is spherically equivalent to
a monomial map, and hence its Hermitian invariant group contains an $n$-torus.
In particular, the group cannot be finite. As a consequence, the degree
of a map whose group is finite must be at least $3$. \end{remark}

\section{additional examples}

First we completely analyze the situation for the four equivalence classes of rational proper maps from
${\mathbb B}^2$ to ${\mathbb B}^3$. 
By a result of Faran ([Fa1]), each such map
is spherically equivalent to precisely one of the following four maps:
$$ (z_1,z_2) \mapsto (z_1,z_2, 0) \eqno (34.1) $$
$$ (z_1,z_2) \mapsto (z_1,z_1z_2, z_2^2) \eqno (34.2) $$
$$ (z_1,z_2) \mapsto (z_1^2 ,\sqrt{2}z_1 z_2, z_2^2) \eqno (34.3) $$
$$ (z_1,z_2) \mapsto (z_1^3, \sqrt{3} z_1 z_2, z_2^3). \eqno (34.4) $$
We write $G_f$ and $\Gamma_f$ for the groups defined in Definition 1.3. Recall that $G_f$ consists
of those automorphisms $\phi$ with $f \circ \phi=f$. Since all these maps are monomials, we may restrict
to unitary $\phi$ in computing $G_f$. Recall that $\Gamma_f$ is the Hermitian invariant group.

\begin{proposition} Let $G_f$ and $\Gamma_f$ denote the invariant groups from Definition 1.3.
\begin{itemize}
\item For the map in (34.1), $G_f$ is trivial and $\Gamma_f = {\rm Aut}({\mathbb B}^2)$. 
\item For the map in (34.2), $G_f$ is trivial and $\Gamma_f = {\bf U}(1) \oplus {\bf U}(1)$.
\item For the map in (34.3), $G_f$ is cyclic of order two,  and $\Gamma_f = {\bf U}(2)$.
\item For the map in (34.4), $G_f$ is cyclic of order three,  generated by the matrix in (35), where $\eta$
is a primitive cube-root of unity:
$$ \begin{pmatrix} \eta & 0 \cr 0 & \eta^2 \end{pmatrix}. \eqno (35)  $$ 
\item For the map in (34.4), $\Gamma_f$ is the subgroup
of ${\bf U}(2)$ generated by the matrices in (36.1) and (36.2):
$$ \begin{pmatrix} e^{i\theta} & 0 \cr 0 & e^{i\phi} \end{pmatrix} \eqno (36.1) $$
$$ \begin{pmatrix} 0 & 1  \cr 1 & 0 \end{pmatrix}. \eqno (36.2)  $$
\end{itemize}
\end{proposition}

\begin{proof} For the map in (34.1), both statements are obvious.

For the map in (34.2), the first statement is easy. The second follows
from Theorem 6.1 (or Corollary 6.1), because the monomials $z_1^2$ and $z_2$ do not occur
in the formula for the map.  

The map in (34.3) is simply $z^{\otimes 2}$. The first statement follows because the map itself
is invariant under only the linear maps $z \mapsto \pm z$. The second statement follows
from Theorem 4.2.

The map $f$ defined by (34.4) is invariant under the transformation in (35), its square, and the identity (its cube), but under no other non-trivial linear map. Hence the first statement holds.
The second statement is a bit subtle. As usual, let $\rho = |z_1|^2 + |z_2|^2 -1$. We compute the Hermitian
form ${\mathcal H}(f)$ and obtain:
$$ (|z_1|^2 + |z_2|^2)^3 + 3 |z_1|^2 |z_2|^2 (1 - |z_1|^2 - |z_2|^2) = (\rho+1)^3 -1 -3 \rho |z_1|^2 |z_2|^2.  \eqno (37)$$
The terms in (37) involving $\rho$ are of course invariant under ${\bf U}(2)$. 
The term $|z_1|^2 |z_2|^2$ is invariant under both the transformations in (36.1) and (36.2).
We must show that they are invariant under no other unitary maps $L$. Let
$$ L(z_1,z_2) = (u_{11} z_1 + u_{12} z_2, u_{21} z_1 + u_{22} z_2) = (w_1,w_2). $$ 
Setting $|w_1|^2 |w_2|^2 = |z_1|^2 |z_2|^2$ forces $u_{11}u_{21} = u_{12}u_{22} = 0$ and
$|u_{11}u_{22} + u_{12}u_{21}|^2 = 1$. Since $L$ is unitary, the only way $L$ can have off-diagonal terms
is if $u_{11}=u_{22}=0$ and $|u_{12}u_{21}|^2=1$. The second statement follows.
\end{proof}

The map in (34.3) is a special case of the following crucial example.

\begin{example} For $m\ge 2$, put $f(z) = z^{\otimes m}$. Let $\eta$ be a primitive $m$-th root of unity.
Then $f$ is invariant under the map $z \mapsto \eta z$. In fact, $G_f$ is
the cyclic subgroup of ${\bf U}(n)$ generated by this map.
We have seen that $\Gamma_f = {\bf U}(n)$ and ${\bf s}(f) = 1$.
\end{example}

Next we give a polynomial example, for $n=2$,
where $\Gamma_f$ is the trivial group. By Theorem 5.1
this map cannot be spherically equivalent to a monomial map.

\begin{example}  Start with
$(z,w) \mapsto (z,w^2,zw)$ and apply a unitary map to get
$$ (c(z+w^2), c(z-w^2), zw) $$
where $c = {1\over \sqrt{2}}$. Then do the same trick on the last two components to get
$$ \left(c(z+w^2), c(c(z-w^2) + zw), c(c(z-w^2) - zw)\right). $$
Finally tensor on the last slot to get the following map $f(z,w):$
$$ \left(c(z+w^2), c(c(z-w^2) + zw), c(c(z-w^2) - zw)z, c(c(z-w^2) - zw)w\right). \eqno (38) $$
Since $f(0)=0$, the group $\Gamma_f$ is a subset of ${\bf U}(2)$. We compute the squared norm $||f(z,w)||^2$.
The only quadratic term is $|z|^2$; hence $\Gamma_f$ is a subgroup
of ${\bf U}(1) \oplus {\bf U}(1)$. 
There are mixed terms of the form $z{\overline w}^2$ and $z{\overline {z w}}$. The invariance of these terms
under a diagonal unitary matrix $U$ forces $U$ to be the identity.  Thus $\Gamma_f$ is trivial.
\end{example}

We briefly return to the other aspect of group-invariant maps from balls.
Let $G$ be a {\bf finite} subgroup
of ${\bf U}(n)$. Then there is a canonical 
nonconstant $G$-invariant polynomial mapping $p:{\mathbb C}^n \to {\mathbb C}^N$ such that
the image of the unit sphere under $p$ lies in a hyperquadric. 
Unless the group is cyclic and represented in one of three particular ways,
the target cannot be a sphere. For a given subgroup $G$,
the target hyperquadric requires sufficiently many eigenvalues of both signs.
Hence there is an interplay between the values of $N,l$ and representation theory.
See [D1], [D5] and their references for an introduction to this topic.
See [F2] for the first paper finding restrictions on the groups arising for
rational proper maps between balls, see [Li]
for the more general result that such groups must be cyclic, and see [G] for results
when the target is a generalized ball. One delightful result in [G] is that the binary icosahedral
group arises as the invariant group of a polynomial map sending the unit sphere $S^3$ to a hyperquadric
defined by a form with $40$ positive and $22$ negative eigenvalues.

Consider the cyclic group generated by the matrix (35), where now $\eta$ is an arbitrary
odd root of unity; suppose $\eta^{2r+1} = 1$.
This group arises as the invariant group of a monomial proper map from ${\mathbb B}^2$
to ${\mathbb B}^N$, where $N=r+2$.  These maps satisfy the sharp degree estimate discussed in Section 2
and they have many other remarkable properties. See the references in [D5]. When $\eta$
is an even root of unity, we obtain an invariant map to ${\mathbb B}^N_1$. 

We close this section with two more examples. 

\begin{example} Using terminology from [DL] we consider a Whitney sequence of proper
polynomial maps $W_k$ from ${\mathbb B}^2$. Each of these maps is essential. Put
$$ W_1(z_1,z_2) = (z_1, z_1z_2, z_2^2). $$
Given $W_k$, define $W_{k+1}$ by tensoring on the last slot of $W_k$. Thus $W_{k+1} = E_{A_k} W_k$: 
$$ W_2(z_1,z_2) = (z_1, z_1z_2, z_1 z_2^2, z_2^3).$$
$$ W_3(z_1,z_2) = (z_1, z_1z_2, z_1 z_2^2, z_1 z_2^3, z_2^4). $$
For each $k$, $\Gamma_{W_k} = {\bf U}(1) \oplus {\bf U}(1)$. Each $W_k$ has source rank $2$ and image rank $k+2$.  Thus, as in the gap 
conjecture of [HJY], there are no gaps in the possible embedding dimensions when 
the source dimension is $2$. When $n\ge 3$, however, the collection of gaps for 
essential maps is strictly larger than for rational maps with minimum embedding 
dimension.
\end{example}

\begin{example} We briefly consider source dimension $3$. The smallest $N$ for which
there is an essential rational proper map to ${\mathbb B}^N$ is $7$.
Put $c= \cos(\theta)$ and $s=\sin(\theta)$. Consider the one-parameter 
family of proper maps $f:{\mathbb B}^3 \to {\mathbb B}^6$ given by   
$$ f(z_1,z_2,z_3) = (z_1,z_2, c z_3, sz_1z_3,s z_2z_3, sz_3^2).$$ 
For $0 < \theta \le {\pi \over 2}$, the Hermitian invariant group satisfies
$\Gamma_f = {\bf U}(2) \oplus {\bf U}(1)$, and the source rank is $2$.
For $\theta =0 $, the group is ${\rm Aut}({\mathbb B}^3)$, and the source rank is $1$. Thus
each of these maps is inessential. By contrast, consider the family of proper maps $g:{\mathbb B}^3 \to {\mathbb B}^7$ defined, for $0 < \theta < {\pi \over 2}$,  by
$$ g(z_1,z_2,z_3) = (cz_1, z_2, s z_1^2, s z_1 z_2, \sqrt{1+s^2} z_1z_3, z_2 z_3, z_3^2). $$ 
The Hermitian invariant group $\Gamma_g$ is ${\bf U}(1 ) \oplus {\bf U}(1) \oplus {\bf U}(1)$. The source
rank is $3$, the image rank is $7$, and the map is essential. Hence a version of the gap conjecture
for essential maps, as mentioned in Example 7.3, must differ from 
the gap conjecture in [HJY]. \end{example}

\section{Additional information about proper maps between balls}

For most CR manifolds $M$ and $M'$, the only CR maps between them are constant, 
and hence not complicated. It is therefore natural to consider situations in
which non-constant maps exist, and then to find restrictions on the maps based on information about the domain and target manifolds. A classical example is the result of Pinchuk that a proper holomorphic
self-map of a strongly pseudoconvex domain in two or more dimensions
must be an automorphism. In particular, for $n\ge 2$, a rational map sending the unit sphere $S^{2n-1}$
to itself must be a linear fractional transformation. By contrast, a rational map sending $S^{2n-1}$
to $S^{2N-1}$ can have arbitrarily large degree if $N$ is sufficiently large.
The so-called degree estimate conjecture stated below suggests a sharp bound on the degree of
such a map. 

The study of proper maps between balls led to CR complexity theory.
For $3\le n \le N \le 2n-2$, Faran [Fa2] showed that a 
(rational) proper map $f: {\mathbb B}^n \to {\mathbb B}^N$ 
is spherically equivalent to the map $z \mapsto (z,0) = z \oplus 0$.
See [Hu] for the same conclusion under weaker boundary regularity assumptions.
In [Fa1] Faran found the four spherical equivalence 
classes of maps from ${\mathbb B}^2$ to ${\mathbb B}^3$. We 
computed both types of invariant groups for
these maps in Proposition 7.1. See also [HJ] for additional information when $N=2n-1$.

Assume $n\ge 2$. If $N<n$, then each holomorphic map from a sphere
to a sphere is constant.
If $N=n\ge 2$, then each proper holomorphic map between balls is an automorphism, and hence of degree $1$. 
If $N\ge 2n$, then there are uncountably many spherical equivalence classes of proper maps between balls.
See [D2], and for a stronger result, see [DL].

For each $n,N$ (with $n\ge 2$) there is a smallest number $c(n,N)$ such that the degree $d$ of 
every rational proper map $f: {\mathbb B}^n \to {\mathbb B}^{N}$ 
is at most $c(n,N)$. The sharp value of $c(n,N)$ is not known, but see [DL2] for the inequality
$$   d \le {N(N-1) \over 2(2n-3)}. $$
For $n=1$, there is no bound, as the maps $z \to z^d$
illustrate. The smallest value of $c(2,N)$ is unknown;
when $n=2$ there are examples of degree $2N-3$, and this value is known to be sharp for monomial maps.
For $n\ge 3$, the smallest value of $c(n,N)$ is also unknown; in this case 
there are examples of degree
$$ d = {N-1 \over n-1}. $$
By [LP] this bound is sharp for monomial maps. The {\bf degree estimate conjecture} states
that these inequalities for $c(n,N)$, sharp for monomials, hold for all rational maps.
The cases $n=2$ and $n=3$ are the most interesting, as phenomena from lower and higher dimensions clash.
A family of group-invariant {\it sharp polynomials} exists in source dimension
$2$.  See Section 7 for more information and see [D5] for additional references.

\begin{remark} The important recent paper [HJY] discusses progress on the {\it gap conjecture}
for proper rational maps between balls. The conjecture states precise values
for the possible embedding dimensions, given the source dimension. 
Our paper allows one to formulate a gap conjecture
for {\it essential} maps. Section 7 includes additional
discussion and two related examples.\end{remark}

By allowing the target dimension to be sufficiently large, proper maps
can be almost arbitrarily complicated. Consider the following result. (See [D4]). 

\begin{theorem} Let ${p \over q}: {\mathbb C}^n \to {\mathbb C}^N$ be a rational function such that
$||{p \over q}||^2 < 1$ on the closed unit ball. Then there is an integer $k$ and a polynomial mapping
$g:{\mathbb C}^n \to {\mathbb C}^k$ such that
${p \oplus g \over q}$ maps the sphere to the sphere.\end{theorem}

A special case of Theorem 8.1 is used several times in the proofs in this paper.
Given a polynomial map $p$, we require a polynomial map $g$ such that $\epsilon p \oplus g$
is a proper map between balls. We can do so because 
$ 1 - \epsilon^2 ||p(z)||^2 > 0$ on the closed ball when $\epsilon^2$ is sufficiently small.

It is also useful to place Theorem 8.1 in the context of CR complexity. It is not possible to bound
either the degree of $g$ or the dimension $k$ in terms of $n$ and the degrees of $p$ and $q$ alone. 
In order to achieve rational proper maps of arbitrary complexity, one must allow the target dimension to
be arbitrarily high. We give one simple example to further illustrate the depth of this result.

\begin{example} Consider the family of polynomials given by $q_a(z) = 1- a z_1 z_2$.
For $|a| < 2$, the polynomial $q$ has no zeroes on the closed ball. If we seek
a proper map, one of whose components is ${cz_1 \over q_a(z)}$ (here $c \ne 0$ is a constant),
then the minimum possible target dimension for this map tends to infinity as $a$ tends to $2$. \end{example}

\section{bibliography}

\medskip

[BEH] M. S. Baouendi, P. Ebenfelt, and X. Huang, Holomorphic mappings between hyperquadrics with small signature difference, Amer. J. Math. 133 (2011), no. 6, 1633-1661.

\medskip

[BH] M. S. Baouendi and X. Huang, Super-rigidity for holomorphic mappings between hyperquadrics with positive signature, J. Diff. Geom. 69 (2005), no. 2, 379-398. 

\medskip

[CS] J. A. Cima and T. J. Suffridge, Boundary behavior of rational
proper maps, Duke Math. J. 60 (1990), no. 1, 135-138

\medskip

[CM] H. Cao and N. Mok, Holomorphic immersions between compact hyperbolic space forms, Invent. Math. 100 (1990) no. 1, 49-62.

\medskip

[D1] J. P. D'Angelo,  Several Complex Variables and the Geometry of
Real Hypersurfaces,
CRC Press, Boca Raton, Fla., 1993.

\medskip

[D2] J. P. D'Angelo,
Proper holomorphic maps between balls of different dimensions,
Michigan Math. J. 35 (1988), no. 1, 83-90.

\medskip

[D3] J. P. D'Angelo, On the classification of rational sphere maps, preprint.
\medskip

[D4] J. P.  D'Angelo, Hermitian analogues of Hilbert's 17-th problem. Adv. Math. 226 (2011), no. 5, 4607-4637. 

\medskip

[D5] J. P. D'Angelo, Invariant CR mappings. Complex analysis, 95-107, Trends Math., Birkh\"auser/Springer Basel AG, Basel, 2010.

\medskip 

[DL] J. P. D'Angelo and J. Lebl, Homotopy equivalence for proper holomorphic mappings, Advances in Math
286 (2016), 160-180.

\medskip

[DL2] J. P. D'Angelo and J. Lebl, 
On the complexity of proper holomorphic mappings between balls, 
Complex Var. Elliptic Equ. 54 (2009), no. 3-4, 187-204.

\medskip

[Fa1]  J. Faran, Maps from the two-ball to the 
three-ball, Invent. Math. 68 (1982), no. 3, 441-475.

\medskip

[Fa2] J. Faran, The linearity of proper holomorphic maps between balls in the low codimension case, J. Diff. Geom. 24 (1986), no. 1, 15-17.

\medskip

[F1] F. Forstneri\v c,
Extending proper holomorphic maps of positive codimension,
Inventiones Math., 95(1989), 31-62.

\medskip

[F2] F. Forstneri\v c, Proper holomorphic maps from balls,
Duke Math. J. 53 (1986), no. 2, 427-441. 

\medskip

[G] D. Grundmeier, Signature pairs for group-invariant Hermitian polynomials,
Internat. J. Math. 22 (2011), no. 3, 311-343. 

\medskip

[HT] K. Hoffman and C. Terp, Compact subgroups of Lie groups and locally compact groups,
Proc. A. M. S., Vol. 120, No. 2 (1994), 623-634.

\medskip

[Hu] X. Huang,  On a linearity problem for proper maps between balls in
complex spaces
of different dimensions, J. Diff. Geom. 51 (1999), no 1, 13-33.

\medskip

[HJ] X. Huang  and S. Ji,
Mapping ${\mathbb B}_n$ into ${\mathbb B}_{2n-1}$, Invent. Math. 145
(2001), 219-250. 

\medskip

[HJY] X. Huang, S. Ji, and W. Yin, On the third gap for proper holomorphic maps between balls,
Math. Ann. 358 (2014), no. 1-2, 115-142. 

\medskip

[JZ] S. Ji and Y. Zhang,
Classification of rational holomorphic maps from $B^2$ into $B^N$ with
degree $2$. Sci. China Ser. A  52 (2009), 2647-2667.

\medskip

[L] J. Lebl, Normal forms, Hermitian operators, and CR maps of spheres
and hyperquadrics, Michigan Math. J. 60 (2011), no. 3, 603-628.

\medskip

[LP] J. Lebl and H. Peters, Polynomials constant on a hyperplane and
CR maps of spheres, Illinois J. Math. 56 (2012), no. 1, 155-175.

\medskip

[Li] D. Lichtblau, Invariant proper holomorphic maps between balls, Indiana Univ. Math. J. 41 (1992), no. 1, 
213-231.

\medskip

[KM]  V. Koziarz and N. Mok, Nonexistence of holomorphic submersion between complex unit balls equivariant with respect to a lattice and their generalization, Amer. J. Math. 132 (2010)  no. 5, 1347-1363.

\medskip

[S] Y.-T. Siu, Some recent results in complex manifold theory related to vanishing theorems for the semipositive case, Arbeitstagung Bonn 1984, Lect. Notes Math. Vol. 1111, Springer-Verlag 1985, Berlin-Heidelberg-New York, pp. 169-192.

\medskip

[Sm] L. Smith, Polynomial invariants of finite groups: a survey of recent developments,
Bull. Amer. Math. Soc. 34 (1997), no. 3, 211-250.

\end{document}